\documentclass[a4paper,draft]{amsproc}
\usepackage{amssymb}
\usepackage{amscd} 
\usepackage{latexsym}
\usepackage{amsmath}
\usepackage{amsfonts}
\usepackage{mathrsfs}

\theoremstyle{plain}
 \newtheorem{theorem}{Theorem}[section]
 \newtheorem{proposition}{Proposition}[section]
 \newtheorem{lemma}{Lemma}[section]
 \newtheorem{corollary}{Corollary}[section]
\newtheorem*{propositionA1}{Proposition A.1}
\newtheorem*{propositionA2}{Proposition A.2}

\theoremstyle{definition}
 \newtheorem{example}{Example}[section]
 \newtheorem{definition}{Definition}[section]

\theoremstyle{remark}
 \newtheorem{remark}{Remark}[section] 
\newtheorem*{remarkA.1}{Remark A.1}

 \numberwithin{equation}{section}

\renewcommand{\leq}{\leqslant}
\renewcommand{\geq}{\geqslant}
\renewcommand{\setminus}{\smallsetminus}
\title[Multidimensional Tauberian Theorems]{Multidimensional Tauberian theorems for vector-valued distributions}



\author[S. Pilipovi\'{c}]{Stevan Pilipovi\'{c}}
\address{Department of Mathematics and Informatics\\
University of Novi Sad\\
Trg Dositeja Obradovi\'ca 4\\
21000 Novi Sad \\
Serbia }
\email{pilipovic@dmi.uns.ac.rs}

\author[J. Vindas]{Jasson Vindas}
\address{Department of Mathematics\\
Ghent University\\
Krijgslaan 281 Gebouw S22\\
B 9000 Gent\\
Belgium}
\email{jvindas@cage.Ugent.be}

\begin{document}

\vspace{18mm}
\setcounter{page}{1}
\thispagestyle{empty}

\begin{abstract}
We  prove several Tauberian theorems for regularizing transforms of vector-valued distributions. The regularizing transform of $f$ is given by the integral transform $M^{f}_{\varphi}(x,y)=(f\ast\varphi_{y})(x),$ $(x,y)\in\mathbb{R}^{n}\times\mathbb{R}_{+}$, with kernel $\varphi_{y}(t)=y^{-n}\varphi(t/y)$. We apply our results to the analysis of asymptotic stability for a class of Cauchy problems, Tauberian theorems for the Laplace transform, the comparison of quasiasymptotics in distribution spaces, and we give a necessary and sufficient condition for the existence of the trace of a distribution on  $\left\{x_0\right\}\times \mathbb R^m$. In addition, we present a new proof of Littlewood's Tauberian theorem.
\end{abstract}

\maketitle


\emph{Dedicated to the memory of Prof. V. S. Vladimirov and Prof. B. I. Zav'yalov.}

\section{Introduction}
\label{wnwi}
Tauberian theory is an important subject which has shown striking usefulness in diverse areas of mathematics such as number theory, harmonic analysis, probability theory, differential equations, and mathematical physics. The one dimensional theory was intensively developed already in the first half of the 20th century and the main representative results from that period were collected in Wiener's work \cite{wiener} and Hardy's monograph \cite{hardy}; a more recent account can be found in Korevaar's book \cite{korevaarbook}. In contrast, the multidimensional Tauberian theory remained dormant until the 1970's. The breakthrough came with the work of the Russian mathematicians Vladimirov, Drozhzhinov, and Zav'yalov \cite{drozhzhinov-z0,vladimirov1}, which led to the incorporation of generalized functions in the scopes of Tauberian theory. Their deep and extensive work resulted in a powerful Tauberian machinery for multidimensional Laplace transforms \cite{vladimirov-d-z1} and greatly contributed toward the foundation of the field of asymptotic analysis of generalized functions.

The present article goes in the direction of their influential work. We shall here significantly improve the Tauberian theorems of Drozhzhinov and Zav'yalov for quasiasymptotics of tempered distributions from \cite{drozhzhinov-z2,drozhzhinov-z3}. 
We point out that our new Tauberian theorems contain as particular instances Meyer's results on wavelet characterizations of weak scaling exponents \cite{meyer} and the Tauberian theorems from \cite{vindas-pilipovic-rakic} for wavelet transforms by the authors and Raki\'{c}.

We are interested in the following class of integral transforms. Fix $\varphi\in\mathcal{S}(\mathbb{R}^{n})$ and set
$\varphi_{y}(\cdot)=y^{-n}\varphi(\cdot/y)$. To a
tempered distribution $f$, we associate the \emph{regulari\-zing transform}, or standard average according to \cite{drozhzhinov-z2,drozhzhinov-z3},
given by the $C^{\infty}$-function
\begin{equation}
\label{wnwieq1}
M^{f}_{\varphi}(x,y):=(f\ast\varphi_{y})(x)\ , \ \ (x,y)\in\mathbb{R}^{n}\times\mathbb{R}_{+}.
\end{equation}
In their seminal work \cite{drozhzhinov-z2,drozhzhinov-z3}, Drozhzhinov and Zav'yalov showed that if one employs a suitable kernel $\varphi$, then the scaling asymptotic behavior $f(\lambda t)\sim c(\lambda) g(t)$, in the sense introduced by Zav'yalov \cite{zavialov} and explained in Subsection \ref{wnwnq} below, with respect to a regularly varying function $c(\lambda)$ can be characterized (up to possible polynomial corrections) in terms of Tauberian theorems involving the angular asymptotic behavior of (\ref{wnwieq1}) plus a Tauberian estimate of the form
\begin{equation}
\label{wnwieq3} \left|M^{f}_\varphi(\lambda x,\lambda y)\right|\leq y^{-k}O(c(\lambda)), \ \ \ \mbox{uniformly for } \left|x\right|^{2}+y^{2}=1.
\end{equation}

We will revisit the problem and obtain optimal results. Our main results are two Tauberian theorems, stated in Section \ref{wnwtt}. Our first task is the identification of the biggest class of kernels $\varphi$ for which these Tauberian type theorems hold. Drozhzhinov and Zav'yalov proved that the Tauberian theorems are valid if $\hat{\varphi}$, the Fourier transform of $\varphi$, satisfies a non-degenerateness requirement; specifically, if it has a Taylor polynomial at the origin that is non-degenerate, in the sense that such a Taylor polynomial does not identically vanish on any ray through the origin. We will identify the biggest class of kernels associated to this Tauberian problem by finding a more general condition of non-degenerateness. It turns out that the structure of the Taylor polynomials of $\hat{\varphi}$ does not play any role in our notion of non-degenerateness.

In Wiener Tauberian theory \cite{wiener} and its many extensions \cite{beurling, drozhzhinov-z1,pilipovic-stankovicWTD,p-s-v} the Tauberian kernels are those whose Fourier transforms do not vanish at any point. In our theory the Tauberian kernels will be those $\varphi$ such that $\hat{\varphi}$ does not identically vanish on any ray through the origin. This is precisely our notion of non-degenerateness, which fully answers the question we just raised above. We mention that the same class of test functions shows up in other contexts (e.g., \cite{hyton-lutz}). 

The second important achievement of our Tauberian theorems is the complete analysis of critical degrees, namely, when the regular variation index of $c(\lambda)$ is a non-negative integer. This analysis was not present in \cite{drozhzhinov-z2,drozhzhinov-z3} nor in Meyer's work \cite{meyer} on pointwise weak scaling exponents. In such a critical case the classes of associate homogeneous and homogeneously bounded functions \cite{vindas3,vindas4}, defined in Subsection \ref{wnwtt}, will appear as natural terms in the polynomial corrections. 

In order to increase the range of applicability of our results, we will consider distributions with values in Banach spaces. Furthermore, as explained in Section \ref{wnwe}, all our results are also valid for distributions with values in more general classes of locally convex spaces. 

The plan of this article is as follows. Section \ref{wnwn} explains the main spaces and asymptotic notions for distributions to be considered in the paper. We give an Abelian proposition in Section \ref{wnwa}. The main section of this paper is Section \ref{wnwtt}, where we state two Tauberian theorems and discuss some important corollaries; the proofs of the Tauberian theorems are postponed to Section \ref{proofs}. We extend in Section \ref{waE} the distribution wavelet analysis from \cite{holschneider} to distribution with values in Banach spaces; such a wavelet analysis is based on the Lizorkin spaces $\mathcal{S}_{0}(\mathbb{R}^{n})$ and $\mathcal{S}'_{0}(\mathbb{R}^{n})$. Many important arguments used in Section \ref{proofs} rely on the wavelet desingularization formula discussed in Section \ref{waE}. In Section \ref{wnwap} we give several applications of our Tauberian theorems. We discuss sufficient conditions for stabilization in time for Cauchy problems related to a class of
parabolic equations, we apply our results to the Laplace transform and give a new proof of Littlewood's Tauberian theorem, and we describe the relation between quasiasymptotics in $\mathcal{D}'$ and $\mathcal{S}'$. Section 8 indicates various useful extensions of our results from previous sections; as an application, we give a necessary and sufficient condition for a tempered distribution $f\in\mathcal{S}'(\mathbb R^n_t\times\mathbb R^m_\xi)$ to have trace at $t=t_0,$ i.e., for the existence of $ f(t_0,\xi)$ in $\mathcal {S}'(\mathbb R^m_{\xi}).$ Finally, the Appendix contains a suitable reformulation of the results from \cite{vindas4}, which will play an essential role in the proofs of our Tauberian theorems. 

\section{Preliminaries}
\label{wnwn}
In this section we collect some notions needed in this article. Let us start by fixing the notation.
The space $E$ denotes a fixed, but arbitrary, Banach space with norm $\left\|\:\cdot\:\right\|$. If $\mathbf{a}:I\to E$ and $T:I\to\mathbb{R}_{+}$, where $I=(0,A)$ (resp. $I=(A,\infty)$) we write $\mathbf{a}(\lambda)=o(T(\lambda))$ as $\lambda\to0^{+}$ (resp. $\lambda\to\infty$) if $\left\|\mathbf{a}(\lambda)\right\|=o(T(\lambda))$. We shall use a similar convention for the big $O$ Landau symbol. Let $\mathbf{v}\in E$, we write $\mathbf{a}(\lambda)\sim T(\lambda) \mathbf{v}$ if $\mathbf{a}(\lambda)=T(\lambda) \mathbf{v}+o(T(\lambda))$. We use the notation $\mathbb{H}^{n+1} = \mathbb{R}^{n} \times\mathbb{R_+}$ for the upper half-space and $\mathbb{S}^{n-1}$ for the unit sphere of $\mathbb{R}^{n}.$

\subsection{Spaces} \label{spaces} The Schwartz spaces $\mathcal{D}(\mathbb{R}^{n}),$ $\mathcal{S}(\mathbb{R}^{n})$, $\mathcal{D}'(\mathbb{R}^{n}),$ $\mathcal{S}'(\mathbb{R}^{n})$ are well known \cite{schwartz1}. We use constants in the Fourier transform as
$\hat{\varphi}(u)=\int_{\mathbb{R}^{n}}\varphi(t)e^{-iu \cdot t}\mathrm{d}t.$
 We will also work with the space
$\mathcal{S}_0(\mathbb{R}^{n})$ of highly time-frequency localized
functions over $\mathbb{R}^{n}$ 
\cite{holschneider}; it is defined as the closed subspace of $\mathcal{S}(\mathbb{R}^{n})$ consisting of  those functions for which all their moments vanish,
i.e., $$\eta\in \mathcal{S}_0(\mathbb{R}^{n})\; \mbox{ if and only
if }\;  \int_{\mathbb{R}^{n}}t^m\eta(t)\mathrm{d}t=0, \;\mbox{ for
all }\; m\in\mathbb{N}^{n}.$$ 
We provide $\mathcal{S}_0(\mathbb{R}^{n})$ with the relative
topology inhered from $\mathcal{S}(\mathbb{R}^{n})$. This
space is also known as the Lizorkin space of test functions. 

Let $\mathcal{A}(\Omega)$ be a topological vector space of test function over
an open subset $\Omega\subseteq\mathbb{R}^{n}$.
We denote by $\mathcal{A}'(\Omega,E)=L_{b}(\mathcal{A}(\Omega),E)$, the space of continuous
linear mappings from $\mathcal{A}(\Omega)$ into $E$ with the topology of uniform convergence over
bounded subsets of $\mathcal{A}(\Omega)$. We are mainly concerned with
the spaces $\mathcal{D}'(\mathbb{R}^{n},E)$, $\mathcal{S}'(\mathbb{R}^{n},E)$, and $\mathcal{S}'_{0}(\mathbb{R}^{n},E)$; see \cite{silva,schwartzv} for vector-valued
distributions. If $f$ is a scalar-valued generalized function and $\mathbf{v}\in E$, we denote by $f\mathbf{v}=\mathbf{v} f$ the $E$-valued
generalized function given
by $\left\langle f(t)\mathbf{v},\varphi(t)\right\rangle=\left\langle f,
\varphi\right\rangle\mathbf{v}$.

 \subsection{Quasiasymptotics}
\label{wnwnq} 
Recall a positive real-valued function, measurable on an interval $(0,A)$ (resp. $(A,\infty)$), is called \emph{slowly varying} at the origin (resp. at
infinity) \cite{bingham,korevaarbook,seneta} if
\begin{equation*}
\lim_{\lambda\rightarrow0^{+}}\frac{L(\lambda a)}{L(\lambda)}=1\ \ \
\left(\mbox{ resp. }
\lim_{\lambda\rightarrow\infty}\right), \ \ \ \mbox{for each }a>0.
\end{equation*}
Throughout the rest of the article, $L$ always stands for a slowly varying function at the origin (resp. infinity).

In the next definition $\mathcal{A}(\mathbb{R}^{n})$ is assumed to be a space of functions on which the dilations are continuous operators. We are primarily concerned with $\mathcal{A}=\mathcal{D},\mathcal{S},\mathcal{S}_{0}$.

\begin{definition}
\label{wnwd2} Let $\mathbf{f}\in\mathcal{A}'(\mathbb{R}^{n},E)$. We say that:
\begin{itemize}
\item [(i)] $\mathbf{f}$ is quasiasymptotically bounded of degree $\alpha\in\mathbb{R}$ at the origin (resp. at infinity) with respect to $L$ in $\mathcal{A}'(\mathbb{R}^{n},E)$ if for each test function $\phi\in\mathcal{A}(\mathbb{R}^{n})$ 
\begin{equation*}
\limsup_{\lambda\to0^{+}}\frac{1}{\lambda^{\alpha}L(\lambda)}\left\|\left\langle
\mathbf{f}\left(\lambda x\right),\phi(x)\right\rangle\right\|<\infty 
\ \ \ \left(\mbox{resp. } \limsup_{\lambda\to\infty}\: \right).
\end{equation*}
We write:
$\mathbf{f}\left(\lambda x\right)=O\left(\lambda^{\alpha}L(\lambda)\right)$ in $\mathcal{A}'(\mathbb{R}^{n},E)$ as $\lambda\to0^{+}$ (resp. $\lambda\to\infty$). 
\item [(ii)] $\mathbf{f}$ has quasiasymptotic behavior of degree $\alpha\in\mathbb{R}$ at the origin (resp. at infinity) with respect to $L$ in $\mathcal{A}'(\mathbb{R}^{n},E)$ if there exists $\mathbf{g}\in\mathcal{A}'(\mathbb{R}^{n},E)$ such that for each test function $\phi\in\mathcal{A}(\mathbb{R}^{n})$ the following limit holds, with respect to the norm of $E$, 
\begin{equation*}
\lim_{\lambda\to0^{+}}\frac{1}{\lambda^{\alpha}L(\lambda)}\left\langle
\mathbf{f}\left( \lambda x\right),\phi(x)\right\rangle=\left\langle \mathbf{g}(x),\phi(x)\right\rangle \in E \ \ \
\left(\mbox{resp. } \lim_{\lambda\to\infty}\:\right) .
\end{equation*}
We write:
\begin{equation}
\label{wnweq2.1}
\mathbf{f}\left(\lambda t\right)\sim\lambda^{\alpha}L(\lambda)\mathbf{g}(t)  \ \ \mbox{in}\ \mathcal{A}'(\mathbb{R}^{n},E)  \ \ \mbox{as}\  \lambda\to 0^{+} 
 \left(\mbox{resp. }\ \lambda\to\infty\:\right).
\end{equation}
\end{itemize}
\end{definition}

We shall also employ the following notation for denoting the quasiasymptotic behavior (\ref{wnweq2.1})
$$
\mathbf{f}\left(\lambda t\right)=\lambda^{\alpha}L(\lambda)\mathbf{g}(t)+o\left(\lambda^{\alpha}L(\lambda)\right) \ \ \ \mbox{in}\ \mathcal{A}'(\mathbb{R}^{n},E),
$$
as $\lambda$ tends to either $0^{+}$ or $\infty$, which has a certain advantage when considering quasiasymptotic expansions. It is easy to show \cite{p-s-v} that $\mathbf{g}$ in (\ref{wnweq2.1}) must be homogeneous with degree of homogeneity $\alpha$ as a generalized function in $\mathcal{A}'(\mathbb{R}^{n},E)$, i.e., $\mathbf{g}(a t)=a^{\alpha}\mathbf{g}(t)$, for all $a\in\mathbb{R}_{+}$. We refer to \cite{drozhzhinov-z4} for an excellent presentation of the theory of multidimensional homogeneous distributions. See the monographs \cite{estrada-kanwal2,p-s-v,vladimirov-d-z1} for extensive studies about asymptotic properties of distributions.

\section{Abelian results}
\label{wnwa}
In order to motivate our Tauberian theorems from the next section, we present here an Abelian result which is essentially due to Drozhzhinov and Zav'yalov \cite{drozhzhinov-z2,drozhzhinov-z3}. Let $\mathbf{f}\in\mathcal{S}'(\mathbb{R}^{n},E)$. As in the Introduction,
we set
\begin{equation*}
M^{\mathbf{f}}_{\varphi}(x,y):=(\mathbf{f}\ast\varphi_{y})(x)\in E, \ \ \ (x,y)\in\mathbb{H}^{n+1},
\end{equation*}
the \emph{regularizing transform} of $\mathbf{f}$ with respect to the test function $\varphi\in\mathcal{S}(\mathbb{R}^{n})$. The Tauberian counterpart of
the following proposition is the main subject of this paper. 

\begin{proposition}
\label{Tqp1} Let $\mathbf{f}\in\mathcal{S}'(\mathbb{R}^{n},E)$ and $\varphi\in\mathcal{S}(\mathbb{R}^{n})$.
\begin{itemize}
\item [(i)] Assume that $\mathbf{f}$ is quasiasymptotically bounded of degree $\alpha$ at the origin (resp. at infinity) with respect to $L$ in $\mathcal{S}'(\mathbb{R}^{n},E)$. Then, there exist $k,l\in\mathbb{N}$, $C>0$ and $\lambda_{0}>0$ such that for all $(x,y)\in\mathbb{H}^{n+1}$
\begin{equation}
\label{TqeqA1}
\left\|M_{\varphi}^\mathbf{f}(\lambda x,\lambda y)\right\|\leq C \lambda^{\alpha}L(\lambda) \left( \frac{1}{y}+y\right)^{k}\left(1+\left|x\right|\right)^{l} , \ \ \  \ \lambda\leq\lambda_{0}\ \ \ 
\left(\text{resp. }\lambda_{0}\leq\lambda\  \right).
\end{equation}
\item [(ii)] If $\mathbf{f}\in\mathcal{S}'(\mathbb{R}^{n},E)$ has quasiasymptotics
$\mathbf{f}\left(\lambda t\right)\sim\lambda^{\alpha}L(\lambda)\mathbf{g}(t)$ in $\mathcal{S}'(\mathbb{R}^{n},E)$ as  $\lambda\to0^{+}$
(resp.  $\lambda\to\infty$), then, for each fixed $(x,y)\in\mathbb{H}^{n+1}$,
\begin{equation}
\label{TqeqA2} \lim_{\lambda\to0^{+}} \frac{1}{\lambda^{\alpha}L\left(\lambda\right)}M_{\varphi}^{\mathbf{f}}(\lambda x,\lambda y)=M^{\mathbf{g}}_{\varphi}(x,y)
\text{ in } E  \ \ \
\left(\mbox{resp. } \lim_{\lambda\to\infty}\right).
\end{equation}
\end{itemize}

\end{proposition}

Part (ii) directly follows by definition. The estimate (\ref{TqeqA1}) from part (i) is obtained from the following proposition
by considering the bounded set $$
\mathfrak{B}=\left\{\frac{1}{\lambda^{\alpha}L(\lambda)}\:\mathbf{f}(\lambda\:
\cdot\:): 0<\lambda\leq \lambda_{0}\right\}\ \
\left(\mbox{resp. } \lambda_{0}\leq\lambda\right).$$

\begin{proposition}
\label{wnwp2} Let 
$\mathfrak{B}\subset\mathcal{S}'(\mathbb{R}^{n},E)$ be a bounded set. Then there exist $k,l$ and
$C>0$ such that
\begin{equation*}
\left\|M_{\varphi}^\mathbf{f}(x,y)\right\|\leq C \left( \frac{1}{y}+y\right)^{k}\left(1+\left|x\right|\right)^{l} , \ \ \ \text{for all } \mathbf{f}\in\mathfrak{B}.
\end{equation*}
\end{proposition}
\begin{proof}
The set $\mathfrak{B}$ is equicontinuous, whence we obtain the existence of
$k_{1}\in\mathbb{N}$ and $C_{1}>0$ such that

$$
\left\|\left\langle \mathbf{f},\rho\right\rangle\right\|\leq C_{1} \sup_{t\in\mathbb{R}^{n}, \left|m\right|\leq k_{1}} \left(1+\left|t\right|\right)^{k_1}\left|\rho^{(m)}(t)\right| , \ \ \ \mbox{for all } \rho\in \mathcal{S}(\mathbb{R}^{n})\mbox{ and } \mathbf{f}\in\mathfrak{B}.
$$
Consequently,
\begin{align*}
\left\|M^{\mathbf{f}}_{\varphi}(x,y)\right\| & =\frac{1}{y^{n}}\left\|\left\langle \mathbf{f}(t),\varphi\left(\frac{x-t}{y}\right) \right\rangle\right\|
\\
&
\leq C_{1}\left(\frac{1}{y}+y\right)^{n+k_1} \sup_{u\in\mathbb{R}^{n}, \left|m\right|\leq k_1} \left(1+\left|x\right|+y\left|u\right|\right)^{k_1}\left|\varphi^{(m)}\left(u\right)\right|
\\
&
\leq
C\left(\frac{1}{y}+y\right)^{n+2k_{1}}(1+\left|x\right|)^{k_{1}}  , \ \ \mbox{ for all } \mathbf{f}\in\mathfrak{B},
\end{align*}
where $C=C_{1}\sup_{u\in\mathbb{R}^{n}, \left|m\right|\leq k_1}
\left(1+\left|u\right|\right)^{k_1}\left|\varphi^{(m)}\left(u\right)\right|.$
\end{proof}

\section{Tauberian theorems for quasiasymptotics -- Main results}
\label{wnwtt}
We now state our main results. Their proofs will be postponed to Section \ref{proofs}. We are interested in the ``converse'' to Proposition \ref{Tqp1}. Naturally, not all test functions will be appropriate for the analysis of this problem. Our Tauberian kernels are precisely those test functions occurring in the following definition.

\begin{definition}
\label{Tqd1}
We say that the test function $\varphi\in\mathcal{S}(\mathbb{R}^{n})$ is non-degenerate if for any $\omega\in\mathbb{S}^{n-1}$ 
the function of one variable $R_{\omega}(r)=\hat{\varphi}(r\omega)\in C^{\infty}[0,\infty)$ is not identically zero, that is,
$
\operatorname*{supp} R_{\omega}\neq\emptyset , \ \ \ \text{for each }\omega\in\mathbb{S}^{n-1} .
$
\end{definition}

 We also need to introduce a class of $E$-valued functions which is of great importance in the study of asymptotic properties of distributions \cite{p-s-v,vindas4}. They will appear in our further consideration. The terminology is from  \cite{vindas1,vindas3,vindas-pilipovic1} (see also de Haan theory in \cite{bingham}).

\begin{definition}
\label{Tqd2} Let  $\mathbf{c}:(0,A)\to E$ (resp. $(A,\infty)\to E$) be a continuous $E$-valued function. We say that:
\begin{enumerate}
\item [(i)] $\mathbf{c}$ is associate asymptotically homogeneous of degree 0 with respect to  $L$ if for some $\mathbf{v}\in E$, as $\lambda\to0^{+}$ (resp. $\lambda\to\infty$),
\begin{equation*}
\mathbf{c}(a\lambda)=\mathbf{c}(\lambda)+L(\lambda)\log a \:\mathbf{v}+o(L(\lambda)), \ \ \ \mbox{for each } a>0.
\end{equation*}

\item [(ii)] $\mathbf{c}$ is asymptotically homogeneously bounded of degree 0 with respect to  $L$ if,  as $\lambda\to0^{+}$ (resp. $\lambda\to\infty$),
$$
\mathbf{c}(a\lambda)=\mathbf{c}(\lambda)+O(L(\lambda)),  \ \ \ \mbox{for each } a>0.
$$
\end{enumerate}
\end{definition}
If $\mathbf{c}$ satisfies either condition (i) or (ii) of Definition \ref{Tqd2}, one can show \cite[Prop. 2.3]{vindas1} that $
\left\|\mathbf{c}(\lambda)\right\|=o(\lambda^{-\sigma})$ as $\lambda\to0^{+}$ (resp. $o(\lambda^{\sigma})$ as $\lambda\to\infty$), for any $\sigma>0$.

We begin with the Tauberian theorem for quasi-asymptotic boundedness.

\begin{theorem}
\label{Tqth1}
Let $\mathbf{f}\in\mathcal{S}'(\mathbb{R}^{n},E)$ and let $\varphi\in\mathcal{S}(\mathbb{R}^{n})$ be non-degenerate. The estimate 
\begin{equation}
\label{Tqeq1}
\limsup_{\lambda\rightarrow 0^{+}}\underset{(x,y)\in\mathbb{H}^{n+1}}{\sup_{\left|x\right|^2+y^2=1}}\frac{y^k}{\lambda^{\alpha}L(\lambda)}\left\|M_{\varphi}^{\mathbf{f}}\left(\lambda
x,\lambda y\right)\right\|<\infty \ \ \ \left(\mbox{resp. } \limsup_{\lambda\rightarrow\infty}\right)
\end{equation}
for some $k\in\mathbb{N}$, implies the existence of an $E$-valued polynomial $\mathbf{P}$ of degree less than $\alpha$ (resp. of the form $\mathbf{P}(t)=\sum_{\alpha<\left|m\right|\leq d}t^{m}\mathbf{w}_{m}$) such that:
\begin{itemize}
\item [(i)] If $\alpha\notin\mathbb{N}$,
$\mathbf{f}-\mathbf{P}$ is quasi-asymptotically bounded of degree $\alpha$ at the origin (resp. at infinity) with respect to $L$ in the space $\mathcal{S}'(\mathbb{R}^{n},E)$.
\item [(ii)] If $\alpha=p\in\mathbb{N}$, there exist asymptotically homogeneously bounded $E$-valued functions $\mathbf{c}_{m}$, $\left|m\right|=p$, of degree 0 with respect to $L$ such that $\mathbf{f}$ has the following asymptotic expansion, as $\lambda\to0^{+}$ (resp. $\lambda\to\infty$),
\begin{equation*}
\mathbf{f}\left(\lambda t\right)=\mathbf{P}(\lambda t)+\lambda^{p}\sum_{\left|m\right|=p}t^{m}\mathbf{c}_{m}(\lambda)+O\left(\lambda^{p}L(\lambda)\right) \ \ \ \mbox{in } \mathcal{S}'(\mathbb{R}^{n},E).
\end{equation*}
\end{itemize}
Moreover, denote by $P_{q}$ the homogeneous terms of the Taylor
polynomials of $\hat{\varphi}$ at the origin, that is,
\begin{equation}
\label{Tqeq2}
P_{q}(u)=\sum_{\left|m\right|=q}\frac{{\hat{\varphi}}^{(m)}(0)u^{m}}{m!},
\ \ \ q\in\mathbb{N}.
\end{equation}
Then, the $E$-valued polynomial $\mathbf{P}$ must satisfy
\begin{equation}
\label{Tqeq3}
P_{q}\left(\frac{\partial}{\partial t}\right)\mathbf{P}=\mathbf{0},\ \ \ \mbox{for all } q\in\mathbb{N},
\end{equation}
and, in case (ii), one can find $l$ such that
\begin{equation}
\label{Tqeqextra} P_{q}\left(\frac{\partial}{\partial t}\right)\mathbf{C}(t,\lambda)=O((1+|t|)^{l}L(\lambda)), \ \ \ \mbox{for all } q\in\mathbb{N},
\end{equation}
uniformly in $t\in\mathbb{R}^{n}$ as $\lambda\to0^{+}$ (resp. $\lambda\to\infty^{+}$), where $\mathbf{C}(t,\lambda)=\sum_{\left|m\right|=p}t^{m}\mathbf{c}_{m}(\lambda)$. 
\end{theorem}

Theorem \ref{Tqth1} yields the ensuing important corollary.

\begin{corollary}
\label{Tqc1}
Assume additionally that $\int_{\mathbb{R}^{n}}\varphi(t)\:\mathrm{d}t\neq0$. Then, the estimate $(\ref{Tqeq1})$ is necessary and sufficient for $\mathbf{f}$ to be quasiasymptotically bounded at the origin (resp. at infinity) of degree $\alpha$ in $\mathcal{S}'(\mathbb{R}^{n})$.
\end{corollary}
\begin{proof}
The sufficiency follows at once from (\ref{Tqeq3}) and (\ref{Tqeqextra}), as $P_{0}(t)$ is a nonzero constant. That (\ref{Tqeq1}) is necessary is a consequence of the Abelian result (Proposition \ref{Tqp1}).
\end{proof}
We now consider the quasiasymptotic behavior of $E$-valued tempered distributions.

\begin{theorem}
\label{Tqth2}
Let $\mathbf{f}\in\mathcal{S}'(\mathbb{R}^{n},E)$ and let $\varphi\in\mathcal{S}(\mathbb{R}^{n})$ be non-degenerate.
Then, the existence of the limits
\begin{equation}
\label{Tqeq4}
\lim_{\lambda\to0^{+}}\frac{1}{\lambda^{\alpha}L(\lambda)}M_{\varphi}^{\mathbf{f}}(\lambda x,\lambda y)=\mathbf{M}_{x,y} \ \ \text{for } (x,y)\in\mathbb{H}^{n+1}\cap\mathbb{S}^{n} \ \ \left(\mbox{resp. } \lim_{\lambda\to\infty}\right)
\end{equation}
and the estimate $(\ref{Tqeq1})$, for some $k\in\mathbb{N}$, imply the existence of $\mathbf{g}\in\mathcal{S}'(\mathbb{R}^{n},E)$, which satisfies $M_{\varphi}^{\mathbf{g}}(x,y)=\mathbf{M}_{x,y}$, and an $E$-valued polynomial $\mathbf{P}$ of degree less than $\alpha$ (resp. of the form $\mathbf{P}(t)=\sum_{\alpha<\left|m\right|\leq d}t^{m}\mathbf{w}_{m}$) such that:
\begin{itemize}
\item [(i)] If $\alpha\notin\mathbb{N}$, $\mathbf{g}$ is homogeneous of degree $\alpha$ and, as $\lambda\to0^{+}$ (resp. $\lambda\to\infty$),
\begin{equation*}
\mathbf{f}\left(\lambda t\right)-\mathbf{P}(\lambda t)\sim\lambda^{\alpha}L(\lambda)\mathbf{g}(t) \ \ \ \mbox{in}\ \mathcal{S}'(\mathbb{R}^{n},E).
\end{equation*}
\item [(ii)]
If $\alpha=p\in\mathbb{N}$, $\mathbf{g}$ is associate
homogeneous of order 1 and degree $p$ (cf. \cite[p. 74]{estrada-kanwal2}, \cite{shelkovich}) satisfying
\begin{equation*}
\mathbf{g}(at)= a^{p}\mathbf{g}(t)+a^{p}\log a \sum _{\left|m\right|=p}t^{m}\mathbf{v}_{m} , \ \ \ \mbox{for each }a>0,
\end{equation*}
for some vectors $\mathbf{v}_{m}\in E$, $\left|m\right|=p$, and there are associate asymptotically homogeneous $E$-valued functions $\mathbf{c}_{m}$, $\left|m\right|=p$, of degree 0 with respect to $L$ such that, as $\lambda\to0^{+}$ (resp. $\lambda\to\infty$),
\begin{equation*}
\mathbf{c}_{m}(a \lambda)=\mathbf{c}(\lambda)+ L(\lambda)\log a \:\mathbf{v}_{m}+o(L(\lambda)), \ \ \ \mbox{for each }a>0,
\end{equation*}
and $\mathbf{f}$ has the following asymptotic expansion in $\mathcal{S}'(\mathbb{R}^{n},E)$
\begin{equation*}
\mathbf{f}\left(\lambda t\right)=\mathbf{P}(\lambda t)+\lambda^{p}L(\lambda)\mathbf{g}(t)+\lambda^{p}\sum_{\left|m\right|=p}t^{m}\mathbf{c}_{m}(\lambda)+o\left(\lambda^{p}L(\lambda)\right),
\end{equation*}
as $\lambda\to0^{+}$ (resp. $\lambda\to\infty$).
\end{itemize}
Furthermore, $\mathbf{P}$ satisfies the equations $(\ref{Tqeq3})$ and, in case (ii), we have, for some $l$, 
\begin{equation}
\label{Tqeq6} P_{q}\left(\frac{\partial}{\partial t}\right)\mathbf{C}(t,\lambda)=o((1+|t|)^{l}L(\lambda)), \ \ \ \mbox{for all } q\in\mathbb{N},
\end{equation}
uniformly in $t\in\mathbb{R}^{n}$ as $\lambda\to0^{+}$ (resp. $\lambda\to\infty^{+}$), where $\mathbf{C}(t,\lambda)=\sum_{\left|m\right|=p}t^{m}\mathbf{c}_{m}(\lambda)$. 

\end{theorem}


\begin{corollary}
\label{Tqc2}
If $\varphi$ satisfies the additional requirement $\int_{\mathbb{R}^{n}}\varphi(t)\:\mathrm{d}t\neq0$, then $(\ref{Tqeq4})$ and $(\ref{Tqeq1})$ are necessary and sufficient for $\mathbf{f}$ to have quasiasymptotic behavior in $\mathcal{S}'(\mathbb{R}^{n},E)$, namely,
\begin{equation*}
 \mathbf{f}\left(\lambda t\right)\sim\lambda^{\alpha}L(\lambda)\mathbf{g}(t) \ \ \ \mbox{in}\ \mathcal{S}'(\mathbb{R}^{n},E) \ \ \mbox{as}\ \lambda\to0^{+} \ \ \ \left(\mbox{resp. }\lambda\to\infty\right).
\end{equation*}
In such a case, $\mathbf{g}$ is completely determined by $M_{\varphi}^{\mathbf{g}}(x,y)=\mathbf{M}_{x,y}$.
\end{corollary}

All the above results have a version for distributions $\mathbf{f}\in\mathcal{D}'(\mathbb{R}^{n},E)$.

\begin{corollary}
\label{Tqth3} In the case of asymptotic behavior at the origin, Theorem \ref{Tqth1}, Corollary \ref{Tqc1}, Theorem \ref{Tqth2}, and Corollary \ref{Tqc2} are valid if one replaces $\mathcal{S}'(\mathbb{R}^{n},E)$ and $\mathcal{S}(\mathbb{R}^{n})$ by $\mathcal{D}'(\mathbb{R}^{n},E)$ and $\mathcal{D}(\mathbb{R}^{n})$ everywhere in the statements. 
\end{corollary}

\begin{proof}
Find $r>0$ such that $\operatorname*{supp} \varphi$ is contained in $B(0,r)$, the Euclidean ball of radius $r$ with center at the origin. Write $\mathbf{f}=\mathbf{f}_{1}+\mathbf{f}_{2}$, where $\mathbf{f}_{1}$ has support in $B(0,3r)$ and $\mathbf{f}=\mathbf{f}_{1}$ on $B(0,2r)$. Clearly, $\mathbf{f}$ and $\mathbf{f}_{1}$ have exactly the same quasiasymptotic properties at the origin in the space $\mathcal{D}'(\mathbb{R}^{n},E)$. On the other hand, since $\mathbf{f}_{2}=0$ on $B(0,2r)$, we have that $M^{\mathbf{f}_{1}}_{\varphi}(x,y)=M^{\mathbf{f}}_{\varphi}(x,y)$ on the region $|x|<r$ and $0<y<1$. The assertions are then obtained by applying the results in $\mathcal{S}'(\mathbb{R}^{n},E)$ to the tempered distribution $\mathbf{f}_{1}$.

\end{proof}
At this point it is worth pointing out that the use of non-degenerate test functions in Theorem \ref{Tqth1} and Theorem \ref{Tqth2} (and their corresponding versions for $\mathcal{D}'(\mathbb{R}^{n},E)$) is absolutely imperative. Clearly, if $\hat{\varphi}$ identically vanishes on a ray through the origin, then there are distributions $\mathbf{f}$ for which $M_{\varphi}^{\mathbf{f}}$ is identically zero and hence for those $\mathbf{f}$ the hypotheses (\ref{Tqeq4}) and (\ref{Tqeq1}) are satisfied for all $\alpha.$ However, among such distributions $\mathbf{f}$, it is easy to find explicit examples for which the conclusions of Theorem \ref{Tqth1} and Theorem \ref{Tqth2} do not hold for a given $\alpha$. 

We end this section with several remarks.

\begin{remark}\label{Tqr1} When $\alpha\notin\mathbb{N}$ in Theorem \ref{Tqth2}, the condition 
$M_{\varphi}^{\mathbf{g}}(x,y)=\mathbf{M}_{x,y}$ uniquely determines $\mathbf{g}$, in view of its homogeneity. 
On the other hand, if $\alpha=p\in\mathbb{N}$, the prescribed values of $M_{\varphi}^{\mathbf{g}}$, in general, can only determine
 $\mathbf{g}$ modulo polynomials which are homogeneous of degree $p$ and satisfy (\ref{Tqeq3}). 
\end{remark}

\begin{remark}
\label{Tqr2}
Observe that if  $\varphi\in\mathcal{S}_{0}(\mathbb{R}^{n})$ in Theorem \ref{Tqth1} and Theorem \ref{Tqth2}, then the converse results are also true: If $\mathbf{f}$ has the asymptotic property stated in (i) or (ii) of Theorem \ref{Tqth1} (resp. Theorem \ref{Tqth2}) for an arbitrary $E$-valued polynomial $\mathbf{P}$, then the regularizing transform $M^{\mathbf{f}}_{\varphi}$ must satisfy (\ref{Tqeq1}) (resp. (\ref{Tqeq1}) and (\ref{Tqeq4})). In fact, this follows at once from the moment vanishing properties of $\varphi$.  The same consideration holds for quasi-asymptotics at the origin if we employ kernels $\varphi$ that have all vanishing moments $\int_{\mathbb{R}^{n}}t^{m}\varphi(t)\:\mathrm{d}t=0$ up to order $|m|\leq \alpha$. Notice also that Theorem \ref{Tqth1} and Theorem \ref{Tqth2} include as particular instances all Tauberian results for the wavelet transform from \cite{vindas-pilipovic-rakic} and Meyer's wavelet characterization of pointwise weak scaling exponents of distributions from \cite{meyer}. We also mention that related results in terms of orthogonal wavelet expansions have been obtained in \cite{pilipovic-teofanov,saneva-vindas,walterw}, such results may be regarded as discretized versions of Theorem \ref{Tqth1} and Theorem \ref{Tqth2}.
\end{remark}

\begin{remark}
\label{Tqr3}
 As mentioned at the Introduction, our motivation in this article comes from the work of Vladimirov, Drozhzhinov and Zav'yalov \cite{drozhzhinov-z2,drozhzhinov-z3, vladimirov-d-z1,vladimirov-d-z2}. In Subsection \ref{LT} below, we shall deduce their so-called general multidimensional Tauberian theorem for the Laplace transform from our Tauberian theorems. Theorem \ref{Tqth1} and Theorem \ref{Tqth2} significantly extend the Tauberian theorems for quasiasymptotics obtained by Drozhzhinov and Zav'yalov in \cite{drozhzhinov-z2,drozhzhinov-z3}. Our notion of non-degenerateness (Definition \ref{Tqd1}) is much less restrictive than the one considered by them.
We say that a polynomial $P$ is non-degenerate (at the origin) if
for each $\omega\in\mathbb{S}^{n-1}$ one has that
$$
P(r\omega)\not\equiv 0, \ \ \ r\in\mathbb{R}_{+}.$$
In their Tauberian theory, Drozhzhinov and Zav'yalov considered the class of test functions $\varphi\in\mathcal{S}(\mathbb{R}^{n})$
for which there exists $N\in\mathbb{N}$ such that
$$T_{\hat{\varphi}}^{N}(u)= \sum_{\left|m\right|\leq N}\frac{\hat{\varphi}^{(m)}(0)u^{m}}{m!},$$ the Taylor polynomial of order $N$ 
at the origin, is non-degenerate. We call such test functions here \emph{strongly non-degenerate}. 
It should be noticed that this type of kernels are included in Definition \ref{Tqd1}; naturally, Definition \ref{Tqd1} gives rise to much 
more kernels. For instance, any non-degenerate $\varphi\in\mathcal{S}_{0}(\mathbb{R}^{n})$ obviously fails to be strongly non-degenerate.
An explicit example of a non-degenerate function $\psi\in\mathcal{S}_{0}(\mathbb{R}^{n})$ is given in the Fourier side by
$\hat{\psi}(u)=e^{-\left|u\right|-(1/\left|u\right|)}.$ Furthermore, if
$\varphi\in\mathcal{S}(\mathbb{R}^{n})$ satisfies
$$
\hat{\varphi}(u)=e^{-\left|u\right|-(1/\left|u\right|)}+u_{1}^{2},\ \ \ \mbox{for } \left|u\right|<1,$$
where $u=(u_{1},u_{2},\dots,u_{n})$, then
$\varphi\notin\mathcal{S}_{0}(\mathbb{R}^{n})$
is a non-degenerate in the sense of Definition \ref{Tqd1}, but the Taylor polynomials of $\hat{\varphi}$ vanish
on the axis $u_{1}=0$; it therefore fails to be strongly non-degenerate.
\end{remark}

\begin{remark}
\label{Tqr4} Theorems \ref{Tqth1} and \ref{Tqth2} may be restated in terms of an interesting class of spaces introduced by Drozhzhinov and Zavialov in \cite{drozhzhinov-z3}. Let $I$ be an ideal of the ring $\mathbb{C}[t_1,t_{2},\dots,t_{n}]$, the (scalar-valued) polynomials over $\mathbb{C}$ in $n$ variables. Define the space $\mathcal{S}_{I}(\mathbb{R}^{n})$ as the subset of $\mathcal{S}(\mathbb{R}^{n})$ consisting of those $\phi$ such that all Taylor polynomials of $\hat{\phi}$ at the origin belong to the ideal $I$. For example, we have $\mathcal{S}_{I}(\mathbb{R}^{n})=\mathcal{S}_{0}(\mathbb{R}^{n})$ if $I=\left\{0\right\}$ or $\mathcal{S}_{I}(\mathbb{R}^{n})=\mathcal{S}(\mathbb{R}^{n})$ if $I=\mathbb{C}[t_1,t_{2},\dots,t_{n}]$. Let $P_{0},\dots,P_{q},\dots$ be a system of homogeneous polynomials where each $P_{q}$ has degree $q$ (some of them may be identically 0). Consider the ideal $I=[P_{0},P_{2},\dots, P_{q},\dots]$, namely, the ideal generated by the $P_{q}$;  then, one can show \cite{drozhzhinov-z3} that $\mathcal{S}_{I}(\mathbb{R}^{n})$ is a closed subspace of $\mathcal{S}(\mathbb{R}^{n})$ and actually $\phi\in\mathcal{S}_{I}(\mathbb{R}^{n})$ if and only if $\int_{\mathbb{R}^{n}} Q(t)\phi(t)\mathrm{d}t =0$ for all polynomial $\mathcal{Q}$ that satisfies the differential equations $P_{q}(\partial/\partial t)Q=0$, $q=0,1\dots .$ When there is $d\in\mathbb{N}$ such that $P_{q}=0$ for $q>d$ one can relax the previous requirement \cite[Lem. A.5]{drozhzhinov-z3} by just asking it to hold for polynomials $\mathcal{Q}$ with degree at most $d$. One can rephrase Theorem \ref{Tqeq1} and Theorem \ref{Tqeq2} as follows. Let $\mathbf{f}\in\mathcal{S}'(\mathbb{R}^{n},E)$ and let $\varphi\in\mathcal{S}(\mathbb{R}^{n})$ be non-degenerate. Take $P_{q}$ as in (\ref{Tqeq2}). For quasiasymptiotics at the origin set $I=[P_{0},P_{2},\dots, P_{[\alpha]}]$ ($I=\mathbb{C}[t_1,t_{2},\dots,t_{n}]$ if $\alpha<0$) and for the case at infinity $I=[P_{p},P_{p+1},\dots]$ where $p$ is the least non-negative integer $\geq\alpha$ (thus, $I=[P_{0},P_{1},\dots]$ if $\alpha<0$). View $\mathbf{f}$ as an element of $ \mathcal{S}'_{I}(\mathbb{R}^{n},E):=L(\mathcal{S}_{I}(\mathbb{R}^{n}),E)$ via its restriction to $\mathcal{S}_{I}(\mathbb{R}^{n})$. Then,
\begin{enumerate}
\item [(a)] $\mathbf{f}$ is quasiasymptotically bounded of degree $\alpha$ with respect to $L$ in the space $\mathcal{S}'_{I}(\mathbb{R}^{n},E)$ if and only if the estimate (\ref{Tqeq1}) holds.
\item [(b)] $\mathbf{f}$ has quasiasymptotic behavior of degree $\alpha$ with respect to $L$ in $\mathcal{S}'_{I}(\mathbb{R}^{n},E)$ if and only if (\ref{Tqeq1}) and (\ref{Tqeq4}) are satisfied.
\end{enumerate}
These assertions follow from (\ref{Tqeq3}), (\ref{Tqeqextra}), and (\ref{Tqeq6}). These forms of the Tauberian theorems are complemented with the following observation. Given any arbitrary system of homogeneous polynomials $P_{0},P_{1},\dots, P_{q},\dots,$ where $P_{q}$ is homogeneous of degree $q$, there is always a non-degenerate test function $\varphi \in \mathcal{S}(\mathbb{R}^{n})$ such that (\ref{Tqeq2}) holds. To show this, we first use the fact that every Stieltjes moment problem has a solution in $\mathcal{S}(\mathbb{R}^{n})$ \cite{estrada1}. Find then $\varphi_{1}\in\mathcal{S}(\mathbb{R}^{n})$ such that the homogeneous terms of the Taylor polynomials at the origin of $\hat{\varphi}_{1}$ coincide with the $P_{q}$. Of course, $\varphi_{1}$ may be degenerate, but there is a constant $C>0$ such $\hat{\varphi}_{1}(u)+C\hat{\psi}(u)>0$ for $1/2<|u|<1$, where $\psi\in\mathcal{S}_{0}(\mathbb{R}^{n})$ is the test function considered in Remark \ref{Tqr3}. The function $\varphi=\varphi_{1}+\psi$ clearly satisfies the requirements.
\end{remark}

\section{Wavelet analysis on $\mathcal{S}'_{0}(\mathbb{R}^{n},E)$}
\label{waE}

In this section we extend the scalar distribution wavelet analysis given in \cite{holschneider} to $E$-valued generalized functions. Such results will play an important role in our proofs of Theorem \ref{Tqth1} and Theorem \ref{Tqth2}. 

By a wavelet we simply mean an element $\psi\in\mathcal{S}_{0}(\mathbb{R}^{n})$ (cf. Subsection \ref{spaces}), and the wavelet transform of $\mathbf{f}\in\mathcal{S}_{0}'(\mathbb{R}^{n},E)$ with respect to $\psi$ is defined as
\begin{equation*}
\mathcal{W}_{\psi}\mathbf{f}(x,y)=\left\langle \mathbf{f}(x+yt),\bar{\psi}(t)\right\rangle=M^{\mathbf{f}}_{\check{\bar{\psi}}} (x,y)\in E, \ \ \ (x,y)\in\mathbb{H}^{n+1},
\end{equation*}
where $\check{\phi}$ denotes reflection about the origin, i.e., $\check{\phi}(t)=\phi(-t)$. When acting on $\mathcal{S}_{0}(\mathbb{R}^{n})$, the wavelet $\mathcal{W}_{\psi}$ has as range a subspace of the space of highly
localized function over $\mathbb H^{n+1}$, denoted \cite{holschneider} as
$\mathcal{S}(\mathbb {H}^{n+1})$ and consisting of those $\Phi\in
C^{\infty}(\mathbb{H}^{n+1})$ for which
$$\sup_{(x,y)\in
\mathbb {H}^{n+1}}\,\left(y+\frac
{1}{y}\right)^{k_{1}}\left(1+\left|x\right|\right)^{k_{2}}\,\left|\frac{\partial^{l}}{\partial y^{l}}\frac{\partial^{m}}{\partial x^{m}}\Phi (x,y)\right|<\infty,$$
for all $k_{1},k_{2},l\in \mathbb{N}$ and $m\in\mathbb{N}^{n}$.
The canonical topology of this space is defined in the standard way \cite{holschneider}. We have that \cite{holschneider,p-r-v} $\mathcal{W}_{\psi}:\mathcal{S}_{0}(\mathbb{R}^{n})\mapsto\mathcal{S}(\mathbb{H}^{n+1})$ is a continuous linear map. We are interested in those wavelets for which $\mathcal{W}_{\psi}$ admits a left inverse. For wavelet-based reconstruction, we shall use the wavelet
synthesis operator \cite{holschneider}. Given $\Phi\in\mathcal{S}(\mathbb{H}^{n+1})$, we define the \textit{wavelet
synthesis operator} with respect to the wavelet $\psi$ as
\begin{equation}
\label{wnwneq6}
\mathcal{M}_{\psi}\Phi(t)=\int^{\infty}_{0}\int_{\mathbb{R}^{n}}\Phi(x,y)\frac{1}{y^{n}}\psi\left(\frac{t-x}{y}\right)\frac{\mathrm{d}x\mathrm{d}y}{y}\: ,
\ \ \ t \in \mathbb{R}^{n}.
\end{equation}
One can show that $\mathcal{M}_{\psi}:\mathcal{S}(\mathbb{H}^{n+1})\rightarrow
\mathcal{S}_0(\mathbb{R}^{n})$ is continuous \cite{holschneider,p-r-v}.

We shall say that the wavelet $\psi\in\mathcal{S}_{0}(\mathbb{R}^{n})$ admits a \textit{reconstruction wavelet} if there exists $\eta\in\mathcal{S}_{0}(\mathbb{R}^{n})$ such that
\begin{equation}
\label{wnwneq7}
c_{\psi,\eta}(\omega)=\int^{\infty}_{0}\overline{\hat{\psi}}(r\omega)\hat{\eta}(r\omega)\frac{\mathrm{d}r}{r}\: , \ \ \ \omega\in\mathbb{S}^{n-1} ,
\end{equation}
is independent of the direction $\omega$; in such a case we set $c_{\psi,\eta}:=c_{\psi,\eta}(\omega).$
The wavelet $\eta$ is
called a reconstruction wavelet for $\psi$.

If $\psi$ admits the reconstruction wavelet $\eta$, one has the reconstruction
formula \cite{holschneider} for the wavelet transform on $\mathcal{S}_{0}(\mathbb{R}^{n})$
\begin{equation}
\label{wnwneq8}
\mathrm{Id}_{\mathcal{S}_0(\mathbb{R}^{n})}=\frac{1}{c_{\psi,\eta}}\mathcal{M}_{\eta}\mathcal{W}_{\psi}.
\end{equation}

We now characterize those wavelets which have a reconstruction wavelet. Actually, the class of non-degenerate test functions from $\mathcal{S}_{0}(\mathbb{R}^{n})$ (cf. Definition \ref{Tqd1}) coincides with that of wavelets admitting reconstruction wavelets.

\begin{proposition}
\label{wnwp3}
Let $\psi\in\mathcal{S}_{0}(\mathbb{R}^{n})$. Then, $\psi$ admits a reconstruction wavelet if and only if it is non-degenerate.
\end{proposition}
\begin{proof} The necessity is clear, for if $\hat{\psi}(r w_{0})$ identically vanishes in the direction of $w_{0}\in\mathbb{S}^{n-1}$, then $c_{\psi,\eta}(w_{0})=0$ (cf. (\ref{wnwneq7})) for any $\eta\in\mathcal{S}_{0}(\mathbb{R}^{n})$. Suppose now that $\psi$ is non-degenerate. As in (\ref{wnwneq7}), we write
$c_{\psi,\psi}(\omega)=\int_{0}^{\infty}|\hat{\psi}(r\omega)|^{2}r^{-1}\mathrm{d}r>0.$ Set
$\varrho(r,w)= \hat{\psi}(rw)/c_{\psi,\psi}(w),$ $(r,w)\in [0,\infty)\times\mathbb{S}^{n-1};$
obviously, if we prove that $\varrho(\left|u\right|, u/\left|u\right|)\in \mathcal{S}(\mathbb{R}^{n})$ and all its partial derivatives vanish at the origin, then $\eta$ given by $\hat{\eta}(u)=\varrho(\left|u\right|, u /\left|u\right|)$ will be a reconstruction wavelet for $\psi$ and actually $c_{\psi,\eta}=1$. By the characterization theorem  of test functions from  $\mathcal{S}(\mathbb{R}^{n})$ in polar coordinates \cite[Prop. 1.1]{drozhzhinov-z4}, the fact $\hat{\eta}\in\mathcal{S}(\mathbb{R}^{n})$ is a consequence of the relations

$$\left.\left(\frac{\partial}{\partial r}\right)^{k} \varrho(r,\omega)\right|_{r=0}=0,\ \ \ k=0,1,\dots \:; $$
the same relations show that all partial derivatives of $\hat{\eta}$ vanish at the origin, and hence $\eta\in\mathcal{S}_{0}(\mathbb{R}^{n})$.
\end{proof}

 In \cite{holschneider}, (\ref{wnwneq8}) was extended to $\mathcal{S}'_0(\mathbb{R}^{n})$ via duality arguments, the main step being the formula
$$
\int^{\infty}_{0}\int_{\mathbb{R}^{n}}\mathcal{W}_{\psi}f(x,y)\Phi(x,y)\frac{\mathrm{d}x\mathrm{d}y}{y}= \left\langle
f(t),\mathcal{M}_{\bar{\psi}}\Phi\:(t)\right\rangle ,
$$
valid for $\Phi \in
\mathcal{S}(\mathbb{H}^{n+1})$ and  $f\in\mathcal{S}'_{0}(\mathbb{R}^{n})$.
It can be easily extended to the $E$-valued case, as the next proposition shows.

\begin{proposition}
\label{wnwp3.4}
 Let $\mathbf{f}\in\mathcal{S}'_{0}(\mathbb{R}^{n},E)$ and $\psi\in\mathcal{S}_{0}(\mathbb{R}^{n})$. Then
\begin{equation}
\label{wnwneq9}
\int^{\infty}_{0}\int_{\mathbb{R}^{n}}\mathcal{W}_{\psi}\mathbf{f}(x,y)\Phi(x,y)\frac{\mathrm{d}x\mathrm{d}y}{y}= \left\langle
\mathbf{f}(t),\mathcal{M}_{\bar{\psi}}\Phi\:(t)\right\rangle,
\end{equation}
for all $\Phi \in
\mathcal{S}(\mathbb{H}^{n+1}).$
\end{proposition}
\begin{proof} The same argument used in Proposition \ref{wnwp2} shows that $$\mathcal{W}_{\psi}:\mathcal{S}'_{0}(\mathbb{R}^{n},E)\mapsto\mathcal{S}'(\mathbb{H}^{n+1},E)$$ is continuous, where we identify the vector-valued function $\mathcal{W}_{\psi}\mathbf{f}$ with the vector-valued distribution given by
$$
\left\langle \mathcal{W}_{\psi}\mathbf{f}(x,y),\Phi(x,y)\right\rangle=\int^{\infty}_{0}\int_{\mathbb{R}^{n}}\mathcal{W}_{\psi}\mathbf{f}(x,y)\Phi(x,y)\frac{\mathrm{d}x\mathrm{d}y}{y}
$$
 The linear map $T:\mathcal{S}'_{0}(\mathbb{R}^{n},E)\mapsto\mathcal{S}'(\mathbb{H}^{n+1},E)$ given by
$
\left\langle (T\mathbf{f})(x,y),\Phi(x,y)\right\rangle=\left\langle \mathbf{f}(t),\mathcal{M}_{\bar{\psi}}\Phi\:(t)\right\rangle ,
$
is continuous as well. Thus, if we show that $\mathcal{W}_{\psi}$ and $T$ coincide on a dense subset of $\mathcal{S}'_{0}(\mathbb{R}^{n},E)$, we would have (\ref{wnwneq9}). The nuclearity of $\mathcal{S}'_{0}(\mathbb{R}^{n})$ implies that $\mathcal{S}'_{0}(\mathbb{R}^{n})\otimes E\subset\mathcal{S}'_{0}(\mathbb{R}^{n},E)$ is dense; thus, it is enough to verify (\ref{wnwneq9}) for $\mathbf{f}=f\mathbf{v}$, where $f\in\mathcal{S}'_{0}(\mathbb{R}^{n})$ and $\mathbf{v}\in E$. Now, the scalar-valued case implies
\begin{equation*}
\left\langle \mathcal{W}_{\psi}(f\mathbf{v})(x,y),\Phi(x,y)\right\rangle =\left\langle \mathcal{W}_{\psi}f(x,y),\Phi(x,y)\right\rangle\mathbf{v}
=\left\langle f(t)\mathbf{v},\mathcal{M}_{\bar{\psi}}\Phi\:(t)\right\rangle,
\end{equation*}
as required.
\end{proof}

We can now extend the wavelet synthesis operator (\ref{wnwneq6}) to $\mathcal{S}'_{0}(\mathbb{H}^{n+1},E)$. This will yield an important ``desingularization formula'' for $E$-valued distributions. Let $\mathbf{K}\in\mathcal{S}'_{0}(\mathbb{H}^{n+1},E)$. We define $\mathcal{M}_{\psi}:\mathcal{S}'_{0}(\mathbb{H}^{n+1},E)\mapsto\mathcal{S}'_{0}(\mathbb{R}^{n},E)$, a continuous linear map, as
\begin{equation*}
\left\langle
\mathcal{M}_{\psi}\mathbf{K}(t),\rho(t)\right\rangle=\left\langle
\mathbf{K}(x,y),\mathcal{W}_{\bar{\psi}}\rho(x,y)\right\rangle, \ \ \ \rho\in\mathcal{S}_{0}(\mathbb{R}^{n}).
\end{equation*}

\begin{proposition}
\label{wnwth1}
Let $\psi\in\mathcal{S}_{0}(\mathbb{R}^{n})$ be non-degenerate and let $\eta\in\mathcal{S}_{0}(\mathbb{R}^{n})$ be a reconstruction wavelet for it. Then,
\begin{equation}
\label{wnwneq11}
\mathrm{Id}_{\mathcal{S}'_0(\mathbb{R}^{n},E)}=\frac{1}{c_{\psi,\eta}}\mathcal{M}_{\eta}\mathcal{W}_{\psi} .
\end{equation}
Furthermore, we have the desingularization formula,
\begin{equation}
\label{wnwneq12}
\left\langle
\mathbf{f}(t),\rho(t)\right\rangle=\frac{1}{c_{\psi,\eta}}\int^{\infty}_{0}\int_{\mathbb{R}^{n}}\mathcal{W}_{\psi}\mathbf{f}(x,y)\mathcal{W}_{\bar{\eta}}\rho(x,y)\frac{\mathrm{d}x\mathrm{d}y}{y} \:,
\end{equation}
valid for all $\mathbf{f}\in\mathcal{S}'_{0}(\mathbb{R}^{n},E)$ and $\rho\in\mathcal{S}_{0}(\mathbb{R}^{n})$.
\end{proposition}
\begin{proof}
If we apply the definition of $\mathcal{M}_{\eta}$, Proposition \ref{wnwp3.4}, (\ref{wnwneq8}), and use the fact that $c_{\psi,\eta}=c_{\bar{\eta},\bar{\psi}}$, we obtain
\begin{align*}\frac{1}{c_{\psi,\eta}}\left\langle \mathcal{M}_{\eta}\mathcal{W}_{\psi}\mathbf{f}(t),\rho(t)\right\rangle
&
=\frac{1}{c_{\psi,\eta}}\int^{\infty}_{0}\int_{\mathbb{R}^{n}}\mathcal{W}_{\psi}\mathbf{f}(x,y)\mathcal{W}_{\bar{\eta}}\rho(x,y)\frac{\mathrm{d}x\mathrm{d}y}{y}
\\
&
\\
&
=\frac{1}{c_{\psi,\eta}}\left\langle \mathcal{W}_{\psi}\mathbf{f}(x,y),\mathcal{W}_{\bar{\eta}}\rho(x,y)\right\rangle
=\left\langle \mathbf{f}, \frac{1}{c_{\bar{\eta},\bar{\psi}}} \mathcal{M}_{\bar{\psi}}\mathcal{W}_{\bar{\eta}}\rho \right\rangle
\\
&
=\left\langle \mathbf{f}(t),\rho(t)\right\rangle.
\end{align*}
So both (\ref{wnwneq11}) and (\ref{wnwneq12}) have been established.
\end{proof}

\section{Proofs of the Tauberian theorems}\label{proofs}
We now proceed to give proofs of the Tauberian theorems \ref{Tqth1} and \ref{Tqth2}. We need a series of lemmas. We start with a technical one.

\begin{lemma}
\label{Tql1} Let $\mathbf{f}\in \mathcal{S}'(\mathbb{R}^{ n},E)$.
The estimate $(\ref{Tqeq1})$ is equivalent to one of the form $(\ref{TqeqA1})$ ($k$ might be a different exponent). If
\begin{equation}
\label{Tqeq6.1}
\lim_{\lambda\to\infty}\frac{1}{\lambda^{\alpha}L(\lambda)}M_{\varphi}^\mathbf{f}(\lambda x,\lambda y)=\mathbf{M}_{x,y} \ \ \ \mbox{in }E
\end{equation}
exists for every $(x,y)\in \mathbb{H}^{n+1}\cap\mathbb{S}^{n}$, so does it for every $(x,y)\in \mathbb{H}^{n+1}$.
\end{lemma}
\begin{proof} We only need to show that (\ref{Tqeq1}) implies (\ref{TqeqA1}). Our assumption is that there are constants $C_{1},\lambda_{0}>0$ such that
\begin{equation*}
\left\|M^{\mathbf{f}}_{\varphi}\left(\lambda
\xi,\lambda\cos\vartheta\right)\right\|< \frac{C_{1}\lambda^{\alpha}L(\lambda)}{(\cos\vartheta)^{k}},
\end{equation*}
for all $\left|\xi\right|^{2}+(\cos\vartheta)^{2}=1$ and  $0<\lambda\leq \lambda_{0}$ (resp. $\lambda_{0}\leq \lambda$).
We can assume that $1+\left|\alpha\right|\leq k$ and $\lambda_{0}< 1$ (resp. $1<\lambda_{0}$). Potter's estimate \cite[p. 25]{bingham} implies that we may assume
\begin{equation}\label{wnweql11}
\frac{L(r\lambda)}{L(\lambda)}< C_{2}\frac{(1+r)^{2}}{r} ,  \ \ \  \text{for }\lambda,\lambda r\in(0,\lambda_{0}] \ (\mbox{resp. } \lambda,\lambda r\in[\lambda_{0},\infty)\:).
\end{equation}
In addition, since $1/L(\lambda)=o(\lambda^{-1})$ as $\lambda\to0^{+}$ (resp. $1/L(\lambda)=o(\lambda)$, as $\lambda\to\infty$) \cite{bingham,seneta}, we can assume
\begin{equation}\label{wnweql22}
\frac{1}{L(\lambda)}<\frac{C_{3}}{\lambda}, \mbox{ for } 0<\lambda\leq \lambda_{0}
\
\left( \mbox{resp. }  \frac{1}{L(\lambda)}<C_{3}\lambda, \mbox{ for } \lambda_{0}\leq \lambda\: \right).
\end{equation}
After this preparation, we are ready to give the proof. For $(x,y)\in\mathbb{H}^{n+1}$ write $x=r\xi$ and $y=r\cos\vartheta$, with $r=\left|(x,y)\right|$. We always keep $\lambda\leq \lambda_{0}$ (resp. $\lambda_{0}\leq \lambda$). If $r\lambda\leq \lambda_{0}$ (resp. $\lambda_{0}\leq r\lambda$), we have that
\begin{align*}
\left\|M^{\mathbf{f}}_{\varphi}\left(\lambda r
\xi,\lambda r\cos\vartheta\right)\right\|&<\frac{C_{1}}{y^{k}}\lambda^{\alpha} L(\lambda r) r^{\alpha+k}
<C_{1}C_{2}\lambda^{\alpha} L(\lambda)\frac{(1+r)^{\alpha+k+1}}{y^{k}}
\\
&
<C_{4}\lambda^{\alpha}L(\lambda)\left(\frac{1}{y}+y\right)^{\alpha+2k+1} \left(1+\left|x\right|\right)^{\alpha+k+1},
\end{align*}
with $C_{4}=2^{\alpha+k+1}C_{1}C_{2}$. We now analyze the case $\lambda_{0}<\lambda r$ (resp. $\lambda r<\lambda_{0}$). Proposition \ref{wnwp2} implies the existence of $k_{1},l_{1}\in\mathbb{N}$, $k_{1}\geq k,$ and $C_{5}$ such that
\begin{align*}
\left\|M^{\mathbf{f}}_{\varphi}\left(\lambda x,\lambda y\right)\right\|&<C_{5}\left(\frac{1}{\lambda y}+\lambda y\right)^{k_{1}}\left(1+\lambda\left|x\right|\right)^{l_{1}}
\\
&< C_{5} \lambda^{\alpha}L(\lambda)\left(\frac{1}{ y}+ y\right)^{k_{1}}\left(1+\left|x\right|\right)^{l_{1}} \frac{1}{\lambda^{\alpha+k_{1}}L(\lambda)}
\\
&\left(\mbox{resp. }< C_{5} \lambda^{\alpha}L(\lambda)\left(\frac{1}{ y}+ y\right)^{k_{1}}\left(1+\left|x\right|\right)^{l_{1}} \frac{\lambda^{k_{1}+l_{1}}}{\lambda^{\alpha}L(\lambda)}\right)
\\
&
<C_{3}C_{5} \lambda^{\alpha}L(\lambda)\left(\frac{1}{ y}+ y\right)^{k_{1}}\left(1+\left|x\right|\right)^{l_{1}} \left(\frac{r}{\lambda_{0}}\right)^{k_{1}+\alpha+1}
\\
&\left(\mbox{resp. }< C_{3}C_{5} \lambda^{\alpha}L(\lambda)\left(\frac{1}{ y}+ y\right)^{k_{1}}\left(1+\left|x\right|\right)^{l_{1}} \left(\frac{\lambda_{0}}{r}\right)^{k_{1}+l_1-\alpha+1}\right)
\\
&
<C_{6}\lambda^{\alpha}L(\lambda)\left(\frac{1}{ y}+ y\right)^{\alpha+2k_{1}+1}\left(1+\left|x\right|\right)^{\alpha+l_{1}+k_{1}+1}
\\
&
\left(\mbox{resp. }<C_{6}\lambda^{\alpha}L(\lambda)\left(\frac{1}{ y}+ y\right)^{2k_{1}+l_1-\alpha+1}\left(1+\left|x\right|\right)^{l_{1}}\right),
\end{align*}
with $C_{6}=C_{3}C_{5}(2/\lambda_{0})^{\alpha+k_{1}+1}$ (resp. $C_{6}=C_{3}C_{5}\lambda_{0}^{k_{1}+l_1-\alpha+1}$).
Therefore, if $C=\max\left\{C_{4},C_{6}\right\}$, $k_{2}>\left|\alpha\right|+2k_{1}+l_1+1$ and $l_{2}>\alpha+l_{1}+k_{1}+1$,
$$
\left\|M^{\mathbf{f}}_{\varphi}\left(\lambda x,\lambda y\right)\right\|< C\lambda^{\alpha}L(\lambda)\left(\frac{1}{ y}+ y\right)^{k_{2}}\left(1+\left|x\right|\right)^{l_{2}},
$$
for all $(x,y)\in\mathbb{H}^{n+1}$ and $0<\lambda\leq \lambda_{0}$ (resp. $\lambda_{0}\leq\lambda$).

 For the second part of the lemma, fix $(x,y)\in\mathbb{H}^{n+1}$ and write it as $(x,y)=(r\xi,r \cos\vartheta)$, where $(\xi,\cos\vartheta)\in\mathbb{H}^{n+1}\cap\mathbb{S}^{n}$. Then, as $\lambda\to0^{+}$ (resp. $\lambda\to\infty$), we have
\begin{align*}\frac{1}{\lambda^{\alpha}L(\lambda)}
M_{\varphi}^\mathbf{f}(\lambda r \xi,\lambda r \cos\vartheta)&=
\frac{L(\lambda r)}{L(\lambda)}r^{\alpha}\left(\frac{1}{(\lambda r)^{\alpha}L(\lambda r)}M_{\varphi}^\mathbf{f}(\lambda r \xi,\lambda r \cos\vartheta)\right)
\\
&
\longrightarrow 1\cdot r^{\alpha} M_{\xi,\cos\vartheta}\ \ \ \text{in } E.
\end{align*}
\end{proof}

Observe that the restriction of $E$-valued tempered distributions to
$\mathcal{S}_0(\mathbb{R}^{n})$ defines a continuous linear projector
from $\mathcal{S}'(\mathbb{R}^{n},E)$ onto $\mathcal{S}'_0(\mathbb{R}^{n},E)$. For 
$\mathbf{f}\in\mathcal{S}'(\mathbb{R}^{n},E)$, we will keep calling by $\mathbf{f}$ its projection onto $\mathcal{S}'_0(\mathbb{R}^{n},E)$. In particular, it makes sense to talk about quasiasymptotics of $\mathbf{f}$ in $\mathcal{S}'_0(\mathbb{R}^{n},E)$ via this restriction projection. The results from Section \ref{waE} are key for the next lemma.

\begin{lemma}
\label{Tql2}  Let $\mathbf{f}\in\mathcal{S}'(\mathbb{R}^{n},E)$ and let $\varphi\in\mathcal{S}(\mathbb{R}^{n})$ be non-degenerate.

\begin{itemize}
\item [(i)] If there exists $k\in\mathbb{N}$ such that the estimate $(\ref{Tqeq1})$ holds, then $\mathbf{f}$ is quasiasymptotically bounded of degree $\alpha$ at the origin (resp. at infinity) with respect to $L$ in the space $\mathcal{S}'_{0}(\mathbb{R}^{n},E)$.
\item [(ii)] If the limit $(\ref{Tqeq6.1})$ exists for each $(x,y)\in \mathbb{H}^{n+1}\cap\mathbb{S}^{n}$ and there is a $k\in\mathbb{N}$ such that the estimate $(\ref{Tqeq1})$ is satisfied, then $\mathbf{f}$ has quasiasymptotic behavior of degree $\alpha$ at the origin (resp. at infinity) with respect to $L$ in the space $\mathcal{S}'_{0}(\mathbb{R}^{n},E)$.
\end{itemize}
\end{lemma}
\begin{proof}
Consider the non-degenerate wavelet $\psi_{1}\in\mathcal{S}_{0}(\mathbb{R}^{n})$ given by $\hat{\psi}_{1}(u)=e^{-\left|u\right|-(1/\left|u\right|)}$. Set $\psi=\check{\bar{\varphi}}\ast\psi_{1}$; then, $\psi\in\mathcal{S}_{0}(\mathbb{R}^{n})$ is also a non-degenerate wavelet. First notice that $\mathcal{W}_{\psi}\mathbf{f}$ is given by
\begin{align*}
\mathcal{W}_{\psi}\mathbf{f}(x,y)&=\left\langle \mathbf{f}(x+yt),(\check{\varphi}\ast\bar{\psi}_{1})(t)\right\rangle
=\left\langle \mathbf{f}(x+yt),\int_{\mathbb{R}^{n}}\bar{\psi}_{1}(u)\varphi(u-t)\mathrm{d}u \right\rangle
\\
&
=\int_{\mathbb{R}^{n}}M^{\mathbf{f}}_{\varphi}(x+yu,y)\bar{\psi}_{1}(u)\mathrm{d}u.
\end{align*}
Find a reconstruction wavelet $\eta\in\mathcal{S}_{0}(\mathbb{R}^{n})$ for $\psi$.

\emph{Part (i)}.
By Lemma \ref{Tql1}, (\ref{Tqeq1}) is equivalent to the estimate (\ref{TqeqA1}) ($k$ might be however a different number). Thus,
\begin{align}
\label{Tqeq6.2}
\left\|M^{\mathbf{f}}_{\varphi}(\lambda x+\lambda yu,\lambda y)\right\|& 
< C \lambda^{\alpha}L(\lambda)\left(y+\frac{1}{y}\right)^{l+k}(1+|x|)^{l}(1+\left|u\right|)^{l}.
\end{align}
In view of the formula for $\mathcal{W}_{\psi}\mathbf{f}$ in terms of $M_{\varphi}$ and $\psi_{1}$, we arrive at the wavelet estimate
\begin{equation}
\label{Tqeq6.3}
\left\|\mathcal{W}_{\psi}\mathbf{f}(\lambda x,\lambda y)\right\|< C_{1}\lambda^{\alpha}L(\lambda)\left(1+|x|\right)^{l}\left(y+\frac{1}{y}\right)^{l+k},
\end{equation}
valid for all  $(x,y)\in\mathbb{H}^{n+1}$ and $\lambda\leq\lambda_{0}$ (resp. $\lambda_{0}\leq\lambda$), where $C_{1}=C\int_{\mathbb{R}^{n}}(1+\left|u\right|)^{l}\left|\bar{\psi}_{1}(u)\right|\mathrm{d}u$.  Let now $\rho\in\mathcal{S}_{0}(\mathbb{R}^{n})$ be arbitrary. Taking into account the desingularization formula (\ref{wnwneq12}) from Proposition \ref{wnwth1} and (\ref{Tqeq6.3}), we conclude that
$$
\left\|\left\langle \mathbf{f}(\lambda t),\rho(t)\right\rangle\right\|\leq \frac{1}{|c_{\psi,\eta}|}\int^{\infty}_{0}\int_{\mathbb{R}^{n}}\left\|\mathcal{W}_{\psi}\mathbf{f}(\lambda x,\lambda y)\right\||\mathcal{W}_{\bar{\eta}}\rho(x,y)|\frac{\mathrm{d}x\mathrm{d}y}{y}= O(\lambda^{\alpha}L(\lambda)),
$$
which shows the result.

\emph{Part (ii)}. By Lemma \ref{Tql1}, the limit (\ref{Tqeq6.1}) exists for each  $(x,y)\in\mathbb{H}^{n+1}$. 
The estimate (\ref{Tqeq6.2}) and the dominated convergence theorem for Bochner integrals ensure that, for each fixed $(x,y)\in\mathbb{H}^{n+1}$,
\begin{align*}
\frac{1}{\lambda^{\alpha}L(\lambda)}\mathcal{W}_{\psi}\mathbf{f}(\lambda x,\lambda y)&
=\int_{\mathbb{R}^{n}}\frac{1}{\lambda^{\alpha}L(\lambda)}M^{\mathbf{f}}_{\varphi}(\lambda x+\lambda yu,\lambda y)\bar{\psi}_{1}(u)\mathrm{d}u
\\
&
\longrightarrow\mathbf{G}(x,y):=\int_{\mathbb{R}^{n}}\mathbf{M}_{x+yu,y}\:\bar{\psi}_{1}(u)\mathrm{d}u,
\end{align*}
as $\lambda\to0^{+}$ (resp. $\lambda\to\infty$). Observe that the function $\mathbf{G}:\mathbb{H}^{n+1}\mapsto E$ is Bochner measurable and, because of (\ref{Tqeq6.3}), 
$$
\left\|\mathbf{G}(x,y)\right\|\leq C_{1}(1+|x|)^{l}\left(y+\frac{1}{y}\right)^{2l+k}, \ \ \ (x,y)\in\mathbb{H}^{n+1}.
$$
Let finally $\rho\in\mathcal{S}_{0}(\mathbb{R}^{n})$ be arbitrary. We can use the wavelet desingularization formula, in combination with the dominated convergence theorem, to deduce
\begin{align*}
\frac{1}{\lambda^{\alpha}L(\lambda)}\left\langle\mathbf{f}(\lambda t),\rho(t)\right\rangle&=\frac{1}{c_{\psi,\eta}}\int^{\infty}_{0}\int_{\mathbb{R}^{n}}\frac{1}{\lambda^{\alpha}L(\lambda)}\mathcal{W}_{\psi}\mathbf{f}(\lambda x,\lambda y)\mathcal{W}_{\bar{\eta}}\rho(x,y)\frac{\mathrm{d}x\mathrm{d}y}{y}
\\
&
\longrightarrow\frac{1}{c_{\psi,\eta}}\int^{\infty}_{0}\int_{\mathbb{R}^{n}}\mathbf{G}(x,y)\mathcal{W}_{\bar{\eta}}\rho(x,y)\frac{\mathrm{d}x\mathrm{d}y}{y}\:,
\end{align*}
as $\lambda\to0^{+}$ (resp. $\lambda\to\infty$), as required.
\end{proof}




We are now in the position to prove Theorem \ref{Tqth1} and Theorem \ref{Tqth2}. The precise relation between quasiasymptotics in $\mathcal{S}'_{0}(\mathbb{R}^{n},E)$ and $\mathcal{S}'(\mathbb{R}^{n},E)$, studied in Appendix A, is crucial for our arguments.

\begin{proof}[Proof of Theorem \ref{Tqth1}] Lemma \ref{Tql2} implies that $\mathbf{f}$ is quasiasymptotically bounded in
 the space $\mathcal{S}'_{0}(\mathbb{R}^{n},E)$. The existence of the $E$-valued polynomial $\mathbf{P}$ and the $\mathbf{c}_{m}$, in case (ii),
 is then a direct consequence of Proposition A.2. The assertion about the degree of
 $\mathbf{P}$ follows from the growth properties of $L$, since slowly varying functions satisfy $L(\lambda)=o(\lambda^{-\alpha}$) as $\lambda\to0^{+}$  (resp. $o(\lambda^{\sigma})$ as $\lambda\to\infty$) for any $\sigma>0$ \cite{bingham,seneta} (in the case (ii) the terms of order $\left|m\right|=p$ can be
 assumed to be absorbed by the $\mathbf{c}_{m}$).  We show (\ref{Tqeq3}) only in the case of asymptotic behavior at infinity; the proof of the case at the origin is completely analogous. Suppose the $E$-valued polynomial has the form
  $$\mathbf{P}(t)=\sum_{\alpha<\left|m\right|\leq d}t^{m}\mathbf{w}_{m}=\sum_{\nu=[\alpha]+1}^{d}\mathbf{Q}_{\nu}(t),$$
   where each $\mathbf{Q}_{\nu}$ is homogeneous of degree $\nu$. Choose $\alpha<\kappa<[\alpha]+1$.
Then, since $L(\lambda)=O(\lambda^{\kappa-\alpha})$ and $\mathbf{c}_{m}(\lambda)=O(\lambda^{\kappa-\alpha})$, we obtain that
$$
\mathbf{f}(\lambda t)=\mathbf{P}(\lambda t)+O(\lambda^{\kappa}) \mbox{ in }\mathcal{S}'(\mathbb{R}^{n},E).
$$
But then, for each fixed $(x,y)\in\mathbb{H}^{n+1}$, the assumption on the size of $M_{\varphi}^{\mathbf{f}}(\lambda x, \lambda y)$  and Lemma \ref{Tql1} imply that
$$
M_{\varphi}^{\mathbf{P}}(\lambda x, \lambda y)
=
\sum_{\alpha<\left|m\right|\leq d} (-\lambda i)^{\left|m\right|}\frac{\partial^{\left|m\right|}}{\partial u^{m}}\left.\left(
e^{i x\cdot u}{\hat{\varphi}}( yu)\right)\right|_{u=0}
\mathbf{w}_{m}=O(\lambda^{\kappa}).
$$
Then, we infer that for each $\alpha<\nu\leq d$ and each $(x,y)\in\mathbb{H}^{n+1}$,
$$
\mathbf{0}=\sum_{\left|m\right|=\nu}\frac{\partial^{\left|m\right|}}{\partial u^{m}}\left.\left(
e^{i x\cdot u}{\hat{\varphi}}( yu)\right)\right|_{u=0}
\mathbf{w}_{m}=i^{\nu}\sum_{q=0}^{\nu}y^{q}(P_{q}\left(-i\partial/\partial x\right)\mathbf{Q}_{\nu})(x).
$$
Thus,
$$
P_{q}\left(\frac{\partial}{\partial x}\right)\mathbf{Q}_{\nu}=\mathbf{0}, \ \ \ \mbox{for all } q,\nu\in\mathbb{N},
$$
as required. It only remains to establish (\ref{Tqeqextra}). Assume that $\alpha=p\in\mathbb{N}$. We keep $0<y<1$. By (\ref{Tqeq1}), (\ref{Tqeq3}), Lemma \ref{Tql1}, and Proposition \ref{wnwp2}, applied to
$$\frac{1}{\lambda^{p}L(\lambda)}\left(\mathbf{f}(\lambda t)-\lambda^{p}\sum
_{\left|m\right|=p}t^{m}\mathbf{c}_{m}(\lambda)-\mathbf{P}(\lambda t)\right),$$
 there are constants $\lambda_{0},C>0$ and
$k,l\in\mathbb{N}$ such that
$$
\left\|\sum_{\left|m\right|=p} (\lambda i)^{p}\frac{\partial^{\left|m\right|}}{\partial u^{m}}\left.\left(
e^{i x\cdot u}{\hat{\varphi}}(yu)\right)\right|_{u=0}
\mathbf{c}_{m}(\lambda)\right\|\leq C\lambda^{p}L(\lambda)\frac{(1+|x|)^{l}}{y^{k}}\:,
$$
for all $(x,y)\in \mathbb{R}^{n}\times(0,1)$ and  $\lambda\leq\lambda_{0}$ (resp. $\lambda_{0}\leq\lambda$), that is,
$$
\left\|\sum_{q=0}^{p}y^{q}P_{q}\left(-i\frac{\partial}{\partial x}\right)\mathbf{C}(x,\lambda)
\right\|\leq CL(\lambda)\frac{(1+|x|)^{l}}{y^{k}}\:.
$$
If we now select $p$ points $0<y_{1}<y_{2}\dots<y_{p}<1$, we
obtain a system of $p+1$ inequalities with Vandermonde matrix
$A=(y^{\nu}_{j})_{j,\nu}$. Multiplying by $A^{-1}$, we convince ourselves of the existence of $C_1$, independent of $x$ and $q$, such that
$$
\left\|P_{q}\left(-i\frac{\partial}{\partial x}\right)\mathbf{C}(x,\lambda)\right\|\leq C_{1}(1+\left|x\right|)^{l}L(\lambda), \ \ \ \mbox{for } \lambda\leq\lambda_{0} \ \  \ (\mbox{resp. } \lambda_{0}\leq\lambda).
$$
This completes the proof.
\end{proof}

We now aboard the proof of Theorem \ref{Tqth2}.

\begin{proof}[Proof of Theorem \ref{Tqth2}] Lemma \ref{Tql2}, under the assumptions (\ref{Tqeq1}) and (\ref{Tqeq4}), implies that $\mathbf{f}$ has quasiasymptotic behavior in the space $\mathcal{S}' _{0}(\mathbb{R}^{n},E)$.
An application of Proposition A.1 yields now the existence of $\mathbf{g}$, $\mathbf{P}$, and the $\mathbf{c}_{m}$ in case (ii). That $\mathbf{P}$ satisfies the equations (\ref{Tqeq3}) actually follows from  Theorem \ref{Tqth1}. The proof of (\ref{Tqeq6}) in case (ii) is similar to that of (\ref{Tqeqextra}) given in the proof of Theorem \ref{Tqth1}, the details are therefore left to the reader.
\end{proof}

\section{Several applications}
\label{wnwap}
In this section we illustrate our ideas with several applications and examples. 
We study in Subsection \ref{wnwPDE} sufficient conditions for stabilization in time of the solution to a class of Cauchy problems. In Subsection \ref{LT} we show how Tauberian theorems for Laplace transforms can be derived from Corollary \ref{Tqc1}. Finally, we prove in Subsection \ref{wnwDS} that for $E$-valued tempered distributions the quasiasymptotics in $\mathcal{D}'(\mathbb{R}^{n},E)$ and $\mathcal{S}'(\mathbb{R}^{n},E)$ are equivalent.

\subsection{Asymptotic stabilization in time for Cauchy problems}
\label{wnwPDE}

 Let  $\Gamma\subseteq\mathbb{R}^{n}$ be a closed convex cone with vertex at the origin. In particular, we may have $\Gamma=\mathbb{R}^{n}$. Let $P$ be a homogeneous polynomial of degree $d$ such 
that $\Re e\:P(iu)<0$ for all $u\in\Gamma\setminus\left\{0\right\}.$ 
We denote \cite{vladimirovbook,vladimirov-d-z1} as $\mathcal{S}'_{\Gamma}\subseteq\mathcal{S}'(\mathbb{R}^{n})$ 
the subspace of distributions supported by $\Gamma$. 

We consider in this subsection the Cauchy problem
\begin{equation}\label{wnweq3.5}
\frac{\partial}{\partial t}U(x,t)=P\left(\frac{\partial}{\partial x}\right)U(x,t), \ \ \ \lim_{t\to0^{+}}U(x,t)=f(x)\ \ \ \mbox{in }\mathcal{S}'(\mathbb{R}^{n}_{x}),
\end{equation}
$$\operatorname*{supp}\hat{f}\subseteq\Gamma, \ \ \ (x,t)\in \mathbb{H}^{n+1},$$
within the class of functions of slow growth over $\mathbb{H}^{n+1}$, that is, solutions $U$ satisfying
$$
\sup_{(x,t)\in \mathbb{H}^{n+1}} \left|U(x,t)\right|\left(t+\frac{1}{t}\right)^{-k_1}\left(1+\left|x\right|\right)^{-k_{1}}<\infty, \mbox{ for some } k_{1},k_{2}\in\mathbb{N}.
$$
One readily verifies that (\ref{wnweq3.5}) has a unique solution. Indeed,
\begin{equation}
\label{eqCP}
U(x,t)=\frac{1}{(2\pi)^{n}}\left\langle \hat{f}(u), e^{ix\cdot u}e^{tP(iu)}\right\rangle=\frac{1}{(2\pi)^{n}}\left\langle \hat{f}(u), 
e^{ix\cdot u}e^{P\left(it^{1/d}u\right)}\right\rangle
\end{equation}
is the sought solution. We shall apply Theorem \ref{Tqth2} to find sufficient geometric conditions for the stabilization in time of the solution to the Cauchy problem (\ref{wnweq3.5}), namely, we study conditions which ensure the existence of a function $T:(A,\infty)\to\mathbb{R}_{+}$ and a constant $\ell\in\mathbb{C}$ such that the following limits exist
\begin{equation}
\label{wnwPDEeq1}
\lim_{t\to\infty}\frac{U(x,t)}{T(t)}=\ell, \ \ \ \mbox{for each }x\in\mathbb{R}^{n}.
\end{equation}

Let $L$ be slowly varying at infinity and $\alpha\in\mathbb{R}$. We shall say that $U$ {stabilizes along $d$-curves} (at infinity), with respect to $\lambda^{\alpha}L(\lambda)$, if the following two conditions hold:
\begin{enumerate}
\item  The following limits exist:
\begin{equation}
\label{wnwPDEeq2}
\lim_{\lambda\to\infty}\frac{U(\lambda x, \lambda^{d} t)}{\lambda^{\alpha}L(\lambda)}=U_{0}(x,t), \ \ \ (x,t)\in\mathbb{H}^{n+1}\cap\mathbb{S}^{n};
\end{equation}
\item There are constants $C\in\mathbb{R}_{+}$ and $l\in\mathbb{N}$ such that
\begin{equation}
\label{wnwPDEeq3}
\left|\frac{U(\lambda x, \lambda^{d} t)}{\lambda^{\alpha}L(\lambda)}\right|\leq \frac{C}{t^{l}}, \ \ \ (x,t)\in\mathbb{H}^{n+1}\cap\mathbb{S}^{n}.
\end{equation}
\end{enumerate}
We have the following result:

\begin{theorem}
\label{wnwth7.1} 
The solution $U$ to the Cauchy problem $(\ref{wnweq3.5})$ stabilizes along $d$-curves, with respect to $\lambda^\alpha L(\lambda)$, if and only if $f$ has quasiasymptotic behavior of degree $\alpha$ at infinity with respect to $L$.
\begin{proof} 
We can find \cite{vladimirov-d-z1} a test
function $\eta\in\mathcal{S}(\mathbb{R}^{n})$ with the property
$\eta(u)=e^{P(iu)},\; u\in\Gamma.$ 
 Setting $\hat{\varphi}=\eta$ and using (\ref{eqCP}), we express $U$ as a regularizing transform,
\begin{equation}
\label{wnweq3.8}U(x,t)=\left\langle f(\xi),\frac{1}{t^{n/d}}\varphi\left(\frac{x-\xi}{t^{1/d}}\right)\right\rangle=M_{\varphi}^{f}(x,y), \ \ \ \mbox{with } y=t^{1/d}.
\end{equation}
Then conditions (\ref{wnwPDEeq2}) and (\ref{wnwPDEeq3}) directly translate into conditions (\ref{Tqeq4}) and (\ref{Tqeq1}), with $M_{x,y}=U_{0}(x,t^{1/d})$ and $k=dl$. Corollary \ref{Tqc2} then yields the desired equivalence.
\end{proof}
\end{theorem}

\begin{corollary}
\label{wnwPDEc1} If $U$ stabilizes along $d$-curves, with respect to $\lambda^{\alpha}L(\lambda)$, then $U$ stabilizes in time with respect to $T(t)=t^{\alpha /d}L(t^{1/d})$. Moreover, the limit $(\ref{wnwPDEeq1})$ holds uniformly for $x$ in compacts of $\mathbb{R}^{n}$.
\end{corollary}
\begin{proof} By Theorem \ref{wnwth7.1}, there exists $g\in\mathcal{S}'(\mathbb{R}^{n})$ such that
$$
f(\lambda \xi)\sim \lambda^{\alpha}L(\lambda)g(\xi) \ \ \ \mbox{as }\lambda\to\infty\ \mbox{in }\mathcal{S}'(\mathbb{R}^{n}).
$$
If $K\subset\mathbb{R}^{n}$ is compact, then,
\begin{align*}
\lim_{t\to\infty}\frac{U(x,t)}{T(t)}&=\lim_{t\to\infty}\frac{1}{t^{\alpha /d}L(t^{1/d})}\left\langle f(t^{1/d}\xi),\varphi\left(\frac{x}{t^{1/d}}-\xi\right)\right\rangle
\\
&
=\left\langle g(\xi),\varphi(-\xi)\right\rangle,
\end{align*}
uniformly for $x\in K$ because 
$\varphi\left(t^{-1/d}x- \xi\right)\to\varphi(-\xi)$ in  $\mathcal{S}(\mathbb{R}^{n}),$ as $t\to\infty.$ 
\end{proof}
\begin{example}[The heat equation]
When $\Gamma=\mathbb{R}^{n}$ and $P(\partial/\partial x)=\Delta$, we obtain that stabilization along parabolas (i.e., $d=2$) is sufficient for stabilization in time of the solution to the Cauchy problem for the heat equation. This particular case of Corollary \ref{wnwPDEc1} was studied in \cite{drozhzhinov-z2,drozhzhinov-z3,drozhzhinov-z5}.
\end{example}

\subsection{Tauberian theorems for Laplace transforms}
\label{LT} We now apply our Tauberian theorems to the  Laplace transform. Throughout this subsection we use the following notation. Let $\Gamma$ be a
closed convex acute cone \cite{vladimirovbook,vladimirov-d-z1}
with vertex at the origin. Its conjugate cone is denoted by $\Gamma^{\ast}$, i.e., 
$$
\Gamma^*=\left\{\xi\in\mathbb{R}^{n}:\: \xi\cdot u\geq 0, \forall u\in\Gamma \right\}.
$$ 
The definition of an acute cone tells us that
$\Gamma^{\ast}$ has non-empty interior, set
$C_{\Gamma}=\operatorname*{int} \Gamma^{\ast}$ and
$T^{C_{\Gamma}}=\mathbb{R}^{n}+iC_{\Gamma}.$ We denote by
$\mathcal{S}'_{\Gamma}(E)$ the subspace of $E$-valued tempered
distributions supported by $\Gamma$. Given
$\mathbf{h}\in\mathcal{S}'_{\Gamma}(E)$, its \emph{Laplace
transform} \cite{vladimirovbook} is
$$
\mathcal{L}\left\{\mathbf{h};z\right\}=\left\langle \mathbf{h}(u),e^{iz\cdot u}\right\rangle, \ \ \ z\in T^{C_{\Gamma}};
$$
it is a holomorphic $E$-valued function on the tube domain
$T^{C_{\Gamma}}$. Fix $\omega\in C_{\Gamma}$. We may write
$\mathcal{L}\left\{\mathbf{h};x+i\sigma\omega\right\}$,
$x\in\mathbb{R}^{n}$, $\sigma>0$, as a regularizing transform. In fact,
choose $\eta_{\omega}\in\mathcal{S}(\mathbb{R}^{n})$ such that
$\eta_{\omega}(u)=e^{- \omega\cdot u},$ $u\in\Gamma. $
Then, 
\begin{equation}
\label{wnweq3.9}
\mathcal{L}\left\{\mathbf{h};x+i\sigma\omega\right\}=M_{\varphi_\omega}^\mathbf{f}(x,\sigma), \ \ \ \mbox{with } 
\hat{\varphi}_{\omega}=\eta_{\omega} \mbox{ and } 
\hat{\mathbf{f}}=(2\pi)^{n}\mathbf{h}.
\end{equation}

The following Tauberian theorems for the Laplace transform were originally obtained in \cite{drozhzhinov-z0,vladimirov-d-z1} under the additional assumption that $C_{\Gamma}$ is a regular cone, i.e., its Cauchy-Szeg\"{o} kernel 
$$K_{C_{\Gamma}}(z)=\int_{\Gamma}e^{i z\cdot u}\mathrm{d}u, \ \ \ z\in T^{C_{\Gamma}},$$ is a divisor of the unity in the Vladimirov algebra $H(T^{C_{\Gamma}})$ \cite{vladimirovbook,vladimirov-d-z1}. We will not make use of such a regularity hypothesis over the cone $\Gamma$.

Given $\kappa\geq0$, we denote by $\Omega^{\kappa}\subset \mathbb{H}^{n+1}$ the set
\begin{equation}
\label{wnwLeq1}\Omega^{\kappa}=\left\{(x,\sigma)\in\mathbb{H}^{n+1}:\:\left|x\right|\leq\sigma^{\kappa}\mbox{ and }0<\sigma\leq1 \right\}.
\end{equation}

\begin{theorem}
\label{wnwLth1} Let $\mathbf{h}\in\mathcal{S}'_{\Gamma}(E)$ and let $L$ be slowly varying at infinity. 
Then, $\mathbf{h}$ is quasiasymptotically bounded of degree $\alpha$ at infinity with respect to $L$ if and only if 
there exist numbers $k\in\mathbb{N}$ and $0\leq\kappa<1$ and a vector $\omega\in C_{\Gamma}$ such that
\begin{equation}
\label{wnwLeq2}
\limsup_{\lambda\to 0^{+}}\underset{\sigma\neq 0}{\sup_{(x,\sigma)\in\partial\Omega^{\kappa}}}
\frac{\sigma^{k}\lambda^{n+\alpha}}{L(1/\lambda)}\left\|\mathcal{L}\left\{\mathbf{h};\lambda\left(x+i\sigma\omega\right)\right\}\right\|<\infty.
\end{equation}
\end{theorem}
\begin{proof} Set $\hat{\mathbf{f}}=(2\pi)^{n}\mathbf{h}$. Clearly, $\mathbf{h}$ is quasiasymptotically bounded of degree $\alpha$ at infinity with respect to $L$ if and only if $\mathbf{f}$ is quasiasymptotically bounded of degree $-\alpha-n$ at the origin with respect to $L(1/\lambda)$. The latter holds, by (\ref{wnweq3.9}) and Corollary \ref{Tqc1}, if and only if there exists $k_{1}\in\mathbb{N}$ such that
\begin{equation}
\label{wnwLeq3}
\limsup_{\lambda\to 0^{+}}\underset{{\vartheta\in[0,\pi/2)}}{\sup_{\left|x\right|^{2}+(\cos \vartheta)^2=1}} 
\frac{\left(\cos\vartheta\right)^{k_1}\lambda^{n+\alpha}}{L(1/\lambda)}\left\|M_{\varphi_{\omega}}^{\mathbf{f}}(\lambda x,\lambda \cos\vartheta)\right\|<\infty.
\end{equation}
Thus, we shall show the equivalence between (\ref{wnwLeq2}) and (\ref{wnwLeq3}). By    Lemma \ref{Tql1}, (\ref{wnwLeq3}) implies (\ref{wnwLeq2}). 
Assume now (\ref{wnwLeq2}), namely, there exist $C_{1}$ and $0<\lambda_{0}<1$ such that
\begin{equation}
\label{wnwLeq4}
\left\|M_{\varphi_\omega}^\mathbf{f}(\lambda x',\lambda \sigma)\right\|< \frac{C_{1}}{\sigma^{k}}\lambda^{-\alpha-n}L\left(1/\lambda\right), 
\ \ \ \lambda\leq\lambda_{0},\ (x',\sigma)\in\Omega^{\kappa}.
\end{equation}
We may assume that $k\geq\alpha+n+1$ and $L$ satisfies
(\ref{wnweql11}) and (\ref{wnweql22}) (the case at infinity). We
keep arbitrary $\lambda<\lambda_{0}$,
$\vartheta\in(0,\pi/2)$ and $x\in\mathbb{R}^{n}$ with
$\left|x\right|^2+(\cos\vartheta)^{2}=1$. Set
$r=\left|x\right|^{\frac{1}{1-\kappa}}/(\cos\vartheta)^{\frac{\kappa}{1-\kappa}},$
$x'=x/r$ and $\sigma=(\cos\vartheta)/r.$ Observe that
$(x',\sigma)\in\partial\Omega^{\kappa}$. Assume first that
$r\lambda\leq\lambda_{0}$, then, in view of
(\ref{wnwLeq4}) and (\ref{wnweql11}),
\begin{align*}
\left\|M_{\varphi_\omega}^\mathbf{f}(\lambda x,\lambda \cos\vartheta)\right\|
&<\frac{C_{1}}{(\cos\vartheta/r)^{k}}(r\lambda)^{-\alpha-n}L\left(1/(r\lambda)\right)
\\
&
\leq4C_{1}C_{2} \lambda^{-\alpha-n}L\left(1/\lambda\right)
(\cos\vartheta)^{ -k-\frac{\kappa}{1-\kappa}(k-\alpha-n+1)  };
\end{align*}
on the other hand, if now $\lambda_{0}<r\lambda$, Proposition \ref{wnwp2}  implies that for some $k_{2}\in\mathbb{N}$, $k_2\leq k$ and $C_{4}>0$,
\begin{align*}
\left\|M_{\varphi_\omega}^\mathbf{f}(\lambda x,\lambda \cos\vartheta)\right\|
&<\frac{C_{4}}{(\lambda\cos\vartheta)^{k_{2}}}
=\frac{C_{4}}{(\cos\vartheta)^{k_2}}\lambda^{-\alpha-n}L\left(1/\lambda\right)\frac{(1/\lambda)^{k_{2}-\alpha-n}}{L(1/\lambda)}
\\
&
<\frac{C_{4}C_{3}}{(\cos\vartheta)^{k_2}}\lambda^{-\alpha-n}L\left(1/\lambda\right)\left(\frac{r}{\lambda_{0}}\right)^{k_{2}+1-\alpha-n}
\\
&
<\frac{C_{4}C_{3}}{\lambda_{0}^{k_{2}+1-\alpha-n}}\lambda^{-\alpha-n}L\left(1/\lambda\right)(\cos\vartheta)^{ -k_2-\frac{\kappa}{1-\kappa}(k_2-\alpha-n+1)  },
\end{align*}
where we have used (\ref{wnweql22}). Therefore, (\ref{wnwLeq3}) is satisfied with $k_{1}\geq k_2+\kappa(k_2-\alpha-n+1)/(1-\kappa)$.
\end{proof}

We obtain as a corollary the so called general Tauberian theorem for Laplace transforms \cite[p. 84]{vladimirov-d-z1}.
\begin{corollary}\label{wnwLc1}
Let $\mathbf{h}\in\mathcal{S}'_{\Gamma}(E)$ and let $L$ be slowly varying at infinity. Then, the estimate $(\ref{wnwLeq2})$, 
for some $k\in\mathbb{N}$, $0\leq\kappa<1$, and $\omega\in C_{\Gamma}$, and the existence of a solid cone $C'\subset C_{\Gamma}$ 
(i.e., $\operatorname*{int}C'\neq \emptyset$) such that
\begin{equation}
\label{wnwLeq5}
\lim_{\lambda\to0^{+}} \frac{\lambda^{\alpha+n}}{L(1/\lambda)}\mathcal{L}\left\{\mathbf{h};i\lambda \xi\right\}=\mathbf{G}(i\xi), \mbox{ in } E, \ \ \ \mbox{for each }\xi\in C',
\end{equation}
are necessary and sufficient for $\mathbf{h}$ to have quasiasymptotic behavior at infinity of degree $\alpha$, i.e.,
$$
\mathbf{h}(\lambda u)\sim \lambda^{\alpha}L(\lambda)\mathbf{g}(u)\ \ \ \mbox{ in } \mathcal{S'}(\mathbb{R}^{n},E)\ \ \ \mbox{as }\lambda\to\infty, \ \ \  \mbox{for some }\mathbf{g}\in\mathcal{S}'_{\Gamma}(E).
$$
In such a case, $\mathbf{G}(z)=\mathcal{L}\left\{\mathbf{g};z\right\}$, $z\in T^{C_{\Gamma}}$.
\end{corollary}
\begin{proof} Recall \cite{vladimirovbook} that $\mathcal{S}'_{\Gamma}(E)$ is canonically isomorphic to $L_{b}(\mathcal{S}(\Gamma),E)$. By the injectivity of the Laplace transform and the uniqueness property of holomorphic functions, the linear span of $\left\{e^{i \xi\cdot u}: \xi\in C'\right\}$ is dense in $\mathcal{S}(\Gamma)$; observe that (\ref{wnwLeq5}) gives precisely convergence of $(\lambda^{-\alpha}/L(\lambda))\mathbf{h}(\lambda\: \cdot)$
over such a dense subset. To conclude the proof, it suffices to apply Theorem \ref{wnwLth1} and the Banach-Steinhaus theorem.
\end{proof}
\begin{example}[Littlewood's Tauberian theorem]\label{wnwapexl} The classical Tauberian theorem of Littlewood \cite{hardy,korevaarbook,littlewood} states that if

\begin{equation} \label{wnwLeq6} \lim_{\varepsilon \to 0^+} \sum_{n = 0}^{\infty}
c_n e^{-\varepsilon n} = \beta
\end{equation}
and if the Tauberian hypothesis $ c_n = O(1/n) $ is satisfied,
then the numerical series is convergent, i.e.,
$\sum_{n = 0}^{\infty}
c_n = \beta.$

We give a quick proof of this theorem based on Corollary \ref{wnwLc1} and a result from \cite{vindas-estrada2}. We first show that $h(u)=\sum_{n=0}^{\infty}c_{n}\delta(u-n)$ has the quasiasymptotic behavior
\begin{equation}\label{wnwLeq7}
h(\lambda u)=\sum_{n=0}^{\infty}c_{n}\delta(\lambda u-n)\sim \beta \frac{\delta(u)}{\lambda} \ \ \ \mbox{as }\lambda\to\infty\ \mbox{in }\mathcal{S}'(\mathbb{R}_{u}).
\end{equation}
Observe that (\ref{wnwLeq5}) is an immediate consequence of (\ref{wnwLeq6}) (here $n=1$, $\alpha=-1$, $L\equiv1$). We verify (\ref{wnwLeq2}) with $\kappa=0$, actually, on the rectangle $\Omega^{0}=[-1,1]\times(0,1]$. Indeed, (\ref{wnwLeq6}) and the Tauberian hypothesis imply that for suitable 
constants $C_{1},C_{2},C_{3},C_{4}>0$, independent of $(x,\sigma)\in\Omega^{0}$,
\begin{align*}
\left|\mathcal{L}\left\{h;\lambda^{-1}(x+i\sigma)\right\}\right|& =\left|\sum_{n=0}^{\infty}c_{n}e^{-\lambda^{-1} \sigma n}e^{i\lambda^{-1} x n}\right|
\\
&
\leq 
C_1+C_2\sum _{n=1}^{\infty} \frac{e^{-\lambda^{-1} \sigma n}}{n}\left|e^{i\lambda^{-1} x n}-1\right|
\\
& <C_{1}+C_{3} \lambda \sum _{n=1}^{\infty}e^{-\lambda^{-1} \sigma n}
<\frac{C_{4}}{\sigma}, \ (x,\sigma)\in\Omega^{0},\ \lambda\geq1.
\end{align*}
Consequently, Corollary \ref{wnwLc1} yields (\ref{wnwLeq7}). Finally, it is well known that (\ref{wnwLeq7}) 
and $c_{n}=O(1/n)$ imply the convergence of the series; in fact, this is true under more general Tauberian 
hypotheses (cf. \cite[Sec. 3]{vindas-estrada2}). We reproduce here a proof of the convergence conclusion for the sake of completeness. Let $\sigma>1$ be arbitrary. Choose 
$\rho\in\mathcal{D}(\mathbb{R})$ such that $0\leq\rho\leq1$, $\rho(u)=1$ for $u\in[0,1]$, and 
$\operatorname*{supp}\rho\subset[-1,\sigma]$, then, evaluation of (\ref{wnwLeq7}) at $\rho$ gives, for some constant $C_{5}$,
\begin{align*}\limsup_{\lambda\to\infty}\left|\sum_{0\leq n\leq\lambda}c_{n}-\beta\right|& \leq \limsup_{\lambda\to\infty}\left|\sum_{\lambda\leq n}c_{n}\rho\left(\frac{n}{\lambda}\right)\right|
 < C_{5}\limsup_{\lambda\to\infty}\sum_{1<\frac{n}{\lambda}<\sigma}\frac{1}{n}\:\rho\left(\frac{n}{\lambda}\right)
 \\
 &
=C_{5}\int_{1}^{\sigma}\frac{\rho(x)}{x}\:\mathrm{d}x<C_{5}(\sigma-1),
\end{align*}
and so, taking $\sigma\to1^{+}$, we conclude $\sum_{n=0}^{\infty}c_{n}=\beta$.
\end{example}
\begin{remark} We refer to the monograph \cite{vladimirov-d-z1} (and references therein) for the numerous applications of Corollary \ref{wnwLc1} in mathematical physics, especially in quantum field theory (see also \cite{vladimirov-z1,vladimirov-z2}). Probabilistic applications can be found in \cite{yakimiv}. Corollary \ref{wnwLc1} can also be  used to easily recover Vladimirov multidimensional generalization \cite{vladimirov1} of the Hardy-Littlewood-Karamata Tauberian theorem (cf. \cite{drozhzhinov-z0,vladimirov-d-z1}). In connection with Example \ref{wnwapexl}, see \cite{estrada-vindasFourierT,estrada-vindasTohoku,estrada-vindasL,vindas-estrada2} for distributional methods in Tauberian theorems for power and Dirichlet series; see \cite{s-vindas,vindas-estradaP} for applications in prime number theory.
\end{remark}

\subsection{Relation between quasiasymptotics in the spaces $\mathcal{D}'(\mathbb{R}^{n},E)$ and $\mathcal{S}'(\mathbb{R}^{n},E)$} \label{wnwDS}

If an $E$-valued tempered distribution has quasiasymptotic behavior in the space $\mathcal{S}'(\mathbb{R}^{n},E)$ then, clearly, it has the same quasiasymptotic behavior in $\mathcal{D}'(\mathbb{R}^{n},E)$. The converse is also well known in the case of scalar-valued distributions, but the truth of this result is less obvious. There have been several proofs of such a converse result and, remarkably, none of them is simple (cf. \cite{meyer,pilipovic3,vindas-estrada6,vindas-pilipovic1} and especially \cite[Lem. 6]{zavialov88} for the general case). We provide a new proof of this fact, which will actually be derived as an easy consequence of the results from Section \ref{wnwtt}. We begin with quasiasymptotic boundedness.

\begin{proposition}
\label{wnwDSp1} Let $\mathbf{f}\in\mathcal{S}'(\mathbb{R}^{n},E)$. If $\mathbf{f}$ is quasiasymptotically bounded of degree $\alpha$ at the origin (resp. at infinity) with respect to $L$ in the space $\mathcal{D}'(\mathbb{R}^{n},E)$, so is $\mathbf{f}$ in the space $\mathbf{f}\in\mathcal{S}'(\mathbb{R}^{n},E)$.
\end{proposition}
\begin{proof} We will show both assertions at 0 and $\infty$ at the same time. The Banach-Steinhaus theorem implies the existence of $\nu\in\mathbb{N}$, $C>0$, and $\lambda_{0}>0$ such that
$$
\left|\left\langle \mathbf{f}(\lambda t),\rho(t)\right\rangle\right|\leq C \lambda^{\alpha}L(\lambda)\sup_{\left|t\right|\leq 1,\: \left|m\right|\leq \nu}\left|\rho^{(m)}(t)\right|, \ \ \ \mbox{for all }\rho\in\mathcal{D}(B(0,3))
$$
and all $0<\lambda<\lambda_{0}$ (resp.  $\lambda_{0}<\lambda$), where $B(0,3)$ is the ball of radius 3. Let now $\varphi\in\mathcal{D}(B(0,1))$ be such that $\int_{\mathbb{R}^{n}}\varphi(t)\mathrm{d}t\neq 0$. If we take $\rho(t)=y^{-n}\varphi(y^{-1}(x-t))$ in the above estimate, where $0<y<1$ and $\left|x\right|\leq1$, we then obtain at once that (\ref{Tqeq1}) is satisfied with $k=\nu+n$, and consequently Corollary
\ref{Tqc1} implies the result.
\end{proof}
Proposition \ref{wnwDSp1}, the Banach-Steinhaus theorem, and the density of $\mathcal{D}(\mathbb{R}^{n})$ in $\mathcal{S}(\mathbb{R}^{n})$ immediately yield what we wanted:
\begin{corollary}
\label{wnwDSc1}
If $\mathbf{f}\in\mathcal{S}'(\mathbb{R}^{n},E)$ has  quasiasymptotic behavior in the space $\mathcal{D}'(\mathbb{R}^{n},E)$, so does $\mathbf{f}$ have the same quasiasymptotic behavior in the space $\mathcal{S}'(\mathbb{R}^{n},E)$.
\end{corollary}

Corollary \ref{wnwDSc1} tells us then that the quasiasymptotics at the origin in $\mathcal{S}'(\mathbb{R}^{n},E)$ is a local property. Indeed, if $\mathbf{f_{1}}=\mathbf{f_2}$ in near 0, then we easily deduce that they have exactly the same quasiasymptotic properties at the origin in $\mathcal{S}'(\mathbb{R}^{n},E)$.

\section{Further extensions}\label{wnwe}
We indicate in this section some useful extensions and variants of the Tauberian results from Section \ref{wnwtt}.

\subsection{Other Tauberian Conditions} \label{wnwOT}
The Tauberian condition (\ref{Tqeq1}), occurring in Theorems
\ref{Tqth1} and \ref{Tqth2}, can be replaced by an estimate of the form (\ref{wnwLeq2}), that is, one may use the boundary of some set $\Omega^{\kappa}$, $0\leq\kappa<1$ (cf. 
(\ref{wnwLeq1})), instead of the upper half sphere $\mathbb{H}^{n+1}\cap\mathbb{S}^{n}$. Specifically, the same argument given in proof of Theorem \ref{wnwLth1} applies to show that (\ref{Tqeq1}) (and hence (\ref{TqeqA1})) is equivalent to
\begin{equation}
\label{TqTauberianregion}
\limsup_{\lambda\rightarrow0^+}\sup_{(x,y)\in\partial\Omega^{\kappa}, \: y>0}\frac{y^k}{\lambda^{\alpha}L(\lambda)}\left\|M^{\mathbf{f}}_{\varphi}\left(\lambda
x,\lambda y\right)\right\|<\infty\ \ \ \left(\mbox{resp. } \limsup_{\lambda\rightarrow\infty}\right)
\end{equation}
for some $0\leq\kappa<1$ and $k\in\mathbb{N}$ (the $k$ may be different numbers).

\subsection{Distributions with values in regular (LB) spaces}\label{wnwDFS}
We now explain that all the results from Sections \ref{wnwa} and \ref{wnwtt} hold if $E$ is a more general locally convex space. We assume below that $E$ is the inductive
limit of an increasing sequence of Banach spaces $\left\{E_n\right\}_{n\in\mathbb N}$ (an (LB) space), that is,
$
E=\mbox{ind}\lim_{n\to\infty}(E_n,||\cdot||_n) \:,
$ where $E_{1}\subset
E_{2}\subset\dots$ and each injection $E_{n}\to E_{n+1}$ is continuous. Particular examples are
$E=\mathcal{S}'(\mathbb{R}^{n}),\mathcal{S}'_{0}(\mathbb{R}^{n}),\mathcal{D}'(Y)$,
where $Y$ is a compact manifold, among many other important spaces
arising in applications.

We start by assuming that $E$ is \emph{regular}, namely, for any bounded set $\mathfrak{B}$ there exists
$n_0\in\mathbb N$ such that $\mathfrak{B}$ is bounded in $E_{n_{0}}$. One can find in \cite[p. 33]{kom} an overview of several conditions which ensure regularity of inductive limits. Under this assumption, our Abelian and Tauberian theorems from Sections \ref{wnwa}--\ref{wnwtt} for $E$-valued distributions are valid if we replace the norm estimates by memberships in bounded subsets of $E$. For instance, a condition such as (\ref{Tqeq1}) should be replaced by one of the form: There exist $k\in\mathbb{N}$, $\lambda_{0}>0$, and a bounded set $\mathfrak{B}\subset E$ such that for $0<\lambda\leq\lambda_0$ (resp. $\lambda\geq\lambda_0$)
\begin{equation}
\label{wnwDFSeq1}
\frac{y^{k}}{\lambda^{\alpha}L(\lambda)}M^{\mathbf{f}}_{\varphi}\left(\lambda
x,\lambda y\right)\in\mathfrak{B},  \ \ \ \left|x\right|^{2}+y^{2}=1\ ;
\end{equation}
and similarly for (\ref{TqeqA1}) and (\ref{TqTauberianregion}). 

As already observed, (\ref{wnwDFSeq1}) is equivalent to an estimate of the form (\ref{Tqeq1}) in some norm $\left\|\:\cdot\:\right\|_{n_0}$, but the determination of $n_{0}$ could be extremely hard to verify in applications and thus such a Tauberian condition would have no value in some concrete situations. It is therefore desirable to have more realistic Tauberian conditions.  We can achieve this if we
use the Mackey theorem \cite[Thm. 36.2]{treves}, because the condition (\ref{wnwDFSeq1}) is then equivalent to the following one: There exists $k\in\mathbb{N}$ such that for each $e^{\ast}\in E'$
\begin{equation}
\label{wnwDFSeq2}
\limsup_{\lambda\rightarrow0^+}\sup_{\left|x\right|^2+y^2=1}\frac{y^k}{\lambda^{\alpha}L(\lambda)}\left|\left\langle e^{\ast},
M^{\mathbf{f}}_{\varphi}\left(\lambda x,\lambda y\right)\right\rangle\right|<\infty \ \ \ \left(\mbox{resp. } \limsup_{\lambda\rightarrow\infty}\right) .
\end{equation}
Therefore, a version of Theorem \ref{Tqth1} with the Tauberian condition (\ref{wnwDFSeq2}) is valid for distributions with values in regular (LB) spaces. Let us point out that the (DFS$^{\ast}$) spaces are of this kind \cite{kom}.

If we now suppose that $E$ is Montel, the limit condition (\ref{Tqeq4}) can be replaced by the equivalent one: The existence of the limits
\begin{equation}
\label{wnwDFSeq3}
\lim_{\lambda\to0^{+}} \frac{1}{\lambda^{\alpha} L(\lambda)}\left\langle e^{\ast},M^{\mathbf{f}}_{\varphi}(\lambda x,\lambda y)\right\rangle\in \mathbb{C} \ \ \ \left(\mbox{resp. } \lim_{\lambda\rightarrow\infty}\right) ,
\end{equation}
for each $e^{\ast}\in E'$ and $(x,y)\in\mathbb{H}^{n+1}\cap \mathbb{S}^{n}$, and likewise for (\ref{TqeqA2}). So we obtain a version of Theorem \ref{Tqth2} in terms of the conditions (\ref{wnwDFSeq2}) and (\ref{wnwDFSeq3}). For example, this case applies for Silva spaces \cite{silvas}, i.e., when the injections $E_{n}\to E_{n+1}$ are compact. Silva spaces are of course the (DFS) spaces (strong duals of Fr\'{e}chet-Schwartz spaces).

 
Let us discuss an example in order to illustrate the ideas of this subsection.

\begin{example} [Fixation of variables]\label{wnwex9.1} Let $f\in\mathcal{S}'(\mathbb{R}^{n}_{t}\times\mathbb{R}^{m}_{\xi})$ and 
$t_{0}\in\mathbb{R}^{n}$. Following \L ojasiewicz \cite{lojasiewicz2}, we say that the variable $t=t_{0}\in\mathbb{R}^{n}$ can be fixed
 in $f(t,\xi)$ if there exists $g\in\mathcal{S}'(\mathbb{R}^{m}_{\xi})$ such that for each $\eta\in\mathcal{S}(\mathbb{R}^{n}_{t}\times\mathbb{R}^{m}_{\xi})$
$$
\lim_{\lambda\to 0^{+}}\left\langle f(t_{0}+\lambda t,\xi),\eta(t,\xi)\right\rangle=\int_{\mathbb{R}^{n}}\left\langle g(\xi),\eta(t,\xi)\right\rangle\mathrm{d}t.
$$
We write $f(t_{0},\xi)=g(\xi)$, {distributionally}. The nuclearity of the Schwartz spaces
implies that $\mathcal{S}'(\mathbb{R}^{n}_{t}\times\mathbb{R}^{m}_{\xi})$ is isomorphic to $\mathcal{S'}(\mathbb{R}_{t}^{n},E)$, 
where $E=\mathcal{S}'(\mathbb{R}^{m}_{\xi})$, a (DFS) space. Actually, the latter tells us that fixation of variables is nothing but the 
notion of \L ojasiewicz point values itself for $E$-valued distributions (cf. \cite{lojasiewicz,estrada-vindasI,vindas-estrada1} 
for \L ojasiewicz point values). Therefore, the (DSF) space-valued version of Corollary \ref{Tqc2} implies that if 
$\varphi\in\mathcal{S}(\mathbb{R}^{n}_{t})$ with $\int_{\mathbb{R}^{n}}\varphi(t)\mathrm{d}t\neq 0$, then the variable $t=t_{0}$ can be fixed 
in $f(t,\xi)$ if and only if there exists $k$ such that for each $\rho\in\mathcal{S}'(\mathbb{R}^{m}_{\xi})$
$$
\limsup_{\lambda\rightarrow0^+}\underset{(x,y)\in\mathbb{H}^{n+1}}{\sup_{\left|x\right|^2+y^2=1}}
y^k \left|\left\langle f\left(t_0+\lambda
x+\lambda yt,\xi\right),\varphi(t)\rho(\xi)\right\rangle\right|<\infty,
$$
and $\lim_{\lambda\to 0^{+}}\left\langle  f\left(t_0+\lambda
x+\lambda yt,\xi\right),\varphi(t)\rho(\xi)\right\rangle \mbox{ exists for all }(x,y)\in\mathbb{H}^{n+1}\cap\mathbb{S}^{n}.$
\end{example}
\begin{remark}It is well known \cite{hormander1} that  the projection
$\pi: \mathbb{R}^{n}_{t}\times\mathbb{R}^{m}\rightarrow
\left\{t_0\right\}\times\mathbb{R}^{m}$, $\pi(t,\xi)=(t_0,\xi),$
defines the pull-back
$$\mathcal{S}'(\mathbb{R}^{n}_{t}\times\mathbb{R}^{m}_{\xi})\ni
f(t,\xi)\rightarrow f(t_{0},\xi):=\pi^*f(\xi)\in
\mathcal{S}'(\mathbb{R}^{m}_{\xi})$$ if the wave front set of $f$
 satisfies $WF(f)\cap\{(t_0,\xi,\eta,0):\:\xi\in\mathbb R^{m}, \eta\in\mathbb R^{n}\}=\emptyset.$ Thus the result given in Example \ref{wnwex9.1} is interesting because we give a necessary and sufficient condition for the existence of this
  pull-back.
\end{remark}

\section*{A. Appendix\\ Relation between quasiasymptotics in $\mathcal{S}'_{0}(\mathbb{R}^{n},E)$ and $\mathcal{S}'(\mathbb{R}^{n},E)$}

\label{ap}
The purpose of this Appendix is to state two propositions which establish the precise connection between quasiasymptotics in the spaces $\mathcal{S}'_{0}(\mathbb{R}^{n},E)$ and $\mathcal{S}'(\mathbb{R}^{n},E)$. Such a relation was crucial for the arguments given in Section \ref{proofs}. Propositions A.1 and A.2 below are multidimensional generalizations of the results from \cite[Sec. 4]{vindas-pilipovic-rakic} and their proofs are based on recent structural theorems from \cite{vindas4}. We assume again that $E$ is a Banach space.

\begin{propositionA1}
Let $\mathbf{f}\in\mathcal{S}'(\mathbb{R}^{n},E)$ have quasiasymptotics behavior of degree $\alpha$ at the origin (resp. at infinity) with respect to $L$ in $\mathcal{S}'_{0}(\mathbb{R}^{n},E)$, i.e., for each $\rho\in\mathcal{S}_{0}(\mathbb{R}^{n})$ the following limit exists
\begin{equation*}
\tag{A.1}\label{wnweqA1}
\lim_{\lambda\to0^{+}}\frac{1}{ \lambda^{\alpha}L(\lambda)}\left\langle \mathbf{f}(\lambda t),\rho(t)\right\rangle \ \ \ \mbox{in } E
\ \ \ \left(\mbox{resp. }\lim_{\lambda\to\infty}\right).
\end{equation*}
Then, there is $\mathbf{g}\in\mathcal{S}'(\mathbb{R}^{n},E)$ such that:
\begin{itemize}
\item [(i)] If $\alpha\notin\mathbb{N}$, $\mathbf{g}$ is homogeneous of degree $\alpha$ and there
exists an $E$-valued polynomial $\mathbf{P}$ such that
\begin{equation*}\tag{A.2}
\mathbf{f}\left(\lambda t\right)-\mathbf{P}(\lambda t)\sim\lambda^{\alpha}L(\lambda)\mathbf{g}(t) \ \ \ \mbox{in}\ \mathcal{S}'(\mathbb{R}^{n},E) \ \ \mbox{as }\lambda\to0^{+} \ \ \ (\mbox{resp. }\lambda\to\infty) .
\end{equation*}
\item [(ii)] If $\alpha=p\in\mathbb{N}$, $\mathbf{g}$ is associate
homogeneous of order 1 and degree $p$ (cf. \cite[p. 74]{estrada-kanwal2}, \cite{shelkovich}) satisfying
\begin{equation*}
\tag{A.3}
\mathbf{g}(at)= a^{p}\mathbf{g}(t)+a^{p}\log a \sum _{\left|m\right|=p}t^{m}\mathbf{v}_{m} , \ \ \ \mbox{for each }a>0,
\end{equation*}
for some vectors $\mathbf{v}_{m}\in E$, $\left|m\right|=p$, and
there exist an $E$-valued polynomial $\mathbf{P}$ and associate asymptotically homogeneous $E$-valued functions $\mathbf{c}_{m}$, $\left|m\right|=p$, of degree 0 with respect to $L$ such that for each $a>0$
\begin{equation*}
\tag{A.4}
\mathbf{c}_{m}(a \lambda)=\mathbf{c}(\lambda)+ L(\lambda)\log a \:\mathbf{v}_{m}+o(L(\lambda)) \ \ \ \mbox{as }\lambda\to0^{+}  \ \ \ \left(\mbox{resp. }\lambda\to\infty\right)
\end{equation*}
and $\mathbf{f}$ has the following quasiasymptotic expansion
\begin{equation*}
\tag{A.5}
\mathbf{f}\left(\lambda t\right)=\mathbf{P}(\lambda t)+\lambda^{p}L(\lambda)\mathbf{g}(t)+\lambda^{p}\sum_{\left|m\right|=p}t^{m}\mathbf{c}_{m}(\lambda)+o\left(\lambda^{p}L(\lambda)\right) \ \ \ \mbox{in } \mathcal{S}'(\mathbb{R}^{n},E)
\end{equation*}

as $\lambda\to0^{+}$ (resp. $\lambda\to\infty$).
\end{itemize}
\end{propositionA1}
\begin{proof}
Let $\mathcal{S}^0(\mathbb{R}^{n})$ be the image under Fourier transform of $\mathcal{S}_{0}(\mathbb{R}^{n})$. Then, $\mathcal{S}^0(\mathbb{R}^{n})$ is precisely the closed subspace of $\mathcal{S}(\mathbb{R}^{n})$ consisting of test functions which vanish at the origin together with their partial derivatives of any order. Thus, if we Fourier transform (\ref{wnweqA1}) and employ the Banach-Steinhaus theorem, we obtain the existence of $\mathbf{h}_{0}\in {\mathcal{S}^0}'(\mathbb{R}^{n},E)$ such that the restriction of $\mathbf{f}$ to $\mathcal{S}^0(\mathbb{R}^{n})$ satisfies
$$
\hat{\mathbf{f}}(\lambda^{-1}u)\sim \lambda^{n+\alpha}L(\lambda)\mathbf{h}_{0}(u) \ \ \ \text{in} \ {\mathcal{S}^{0}}'(\mathbb{R}^{n},E) 
$$
as $\lambda\to0^{+}$ (resp. $\lambda\to\infty$). The result then follows from \cite[Thm. 3.1]{vindas4}, after taking Fourier inverse transform.
\end{proof}

The proof of the following proposition is completely analogous to that of Proposition A.1, but now making use of \cite[Thm. 3.2]{vindas4} instead of \cite[Thm. 3.1]{vindas4}; we therefore omit it.
\begin{propositionA2}
Let $\mathbf{f}\in\mathcal{S}'(\mathbb{R}^{n},E)$ be quasiasymptotically bounded of degree $\alpha$ at the origin (resp. at infinity) with respect to $L$ in $\mathcal{S}'_{0}(\mathbb{R}^{n},E)$.
Then:
\begin{itemize}
\item [(i)] If $\alpha\notin\mathbb{N}$, there exists an $E$-valued polynomial $\mathbf{P}$ such that
$\mathbf{f}-\mathbf{P}$ is quasiasymptotically bounded of degree $\alpha$ at the origin (resp. at infinity) with respect to $L$ in the space $\mathcal{S}'(\mathbb{R}^{n},E)$.
\item [(ii)] If $\alpha=p\in\mathbb{N}$, there exist an $E$-valued polynomial $\mathbf{P}$ and asymptotically homogeneously bounded $E$-valued functions $\mathbf{c}_{m}$, $\left|m\right|=p$, of degree 0 with respect to $L$ such that $\mathbf{f}$ has the following quasiasymptotic expansion
\begin{equation*}
\tag{A.7}
\mathbf{f}\left(\lambda t\right)=\mathbf{P}(\lambda t)+\lambda^{p}\sum_{\left|m\right|=p}t^{m}\mathbf{c}_{m}(\lambda)+O\left(\lambda^{p}L(\lambda)\right) \ \ \ \mbox{in } \mathcal{S}'(\mathbb{R}^{n},E)
\end{equation*}

as $\lambda\to0^{+}$ (resp. $\lambda\to\infty$).
\end{itemize}
\end{propositionA2}

\bibliographystyle{amsplain}

\end{document}